\documentclass[12pt,reqno]{amsart}
\usepackage{amssymb}
\usepackage[latin1]{inputenc}
\usepackage{dsfont}
\usepackage{hyperref}
\usepackage{fullpage}

\newtheorem{theorem}{Theorem}[section]
\newtheorem{lemma}[theorem]{Lemma}
\newtheorem{proposition}[theorem]{Proposition}
\newtheorem{corollary}[theorem]{Corollary}
\theoremstyle{definition}
\newtheorem{definition}[theorem]{Definition}

\theoremstyle{remark}
\newtheorem{remark}[theorem]{Remark}

\numberwithin{equation}{section}

\begin{document}

\title[The distribution of  eigenvalues of randomizes permutation matrices]{The distribution of  eigenvalues of randomized permutation matrices}

\author[J. Najnudel]{Joseph Najnudel}
\address{Institut f\"ur Mathematik, Universit\"at Z\"urich, Winterthurerstrasse 190,
8057-Z\"urich, Switzerland}
\email{\href{mailto:joseph.najnudel@math.uzh.ch}{joseph.najnudel@math.uzh.ch}}

\author[A. Nikeghbali]{Ashkan Nikeghbali}
\email{\href{mailto:ashkan.nikeghbali@math.uzh.ch}{ashkan.nikeghbali@math.uzh.ch}}

\date{\today}

\begin{abstract}
In this article we study in detail a family of random matrix ensembles which are obtained from random permutations matrices (chosen at random according to the Ewens measure of parameter $\theta>0$) by replacing 
the entries equal to one by more general non-vanishing complex random variables. For these ensembles, in contrast with 
more classical models as the Gaussian Unitary Ensemble, or the Circular Unitary Ensemble,  the  eigenvalues can be very explicitly computed by using the cycle structure of the permutations. Moreover, by using the so-called 
virtual permutations, first introduced by Kerov, Olshanski and Vershik, and studied with a probabilistic point of view by Tsilevich, we are able 
to define, on the same probability space, a model for each dimension greater than or equal to one, which gives a meaning to the notion of almost sure convergence when
the dimension tends to infinity. In the present paper, depending on the precise model which is considered, we obtain a number of different results of convergence for the point measure of the eigenvalues, some of these results giving a strong convergence, which is not common in random matrix theory.
\end{abstract}

\maketitle
\tableofcontents
\section{Introduction}
\subsection{Random permutation matrices and outline of the paper}
The distribution of the eigenvalues of random matrices and some related objects such as their characteristic polynomials have received much attention in the last few decades. They have been applied in such diverse branches as physics, number theory, analysis or probability theory as illustrated in the monographs \cite{ms}, \cite{mehta} and \cite{agz} or the survey paper \cite{diaconis}. The main matrix ensembles which have been studied are the Gaussian ensembles and  some of their generalizations, the classical compact Lie groups $U(N)$ (the group of unitary matrices), $O(N)$ (the orthogonal group) and some of their continuous sub-groups, endowed with the Haar probability measure. It is also natural to investigate for the distribution of the eigenvalues  of random permutation matrices, i.e. matrices which are canonically associated to a random element of a given
finite symmetric group. Indeed it is well-known that the eigenvalues of a permutation matrix $M_\sigma$ associated with a permutation $\sigma$ are entirely determined by the cycle structure of $\sigma$, and hence one can hope to take advantage of the extensive literature on random permutations (see e.g. the book by Arratia, Barbour and Tavar\'e \cite{ABT}) to describe completely the structure of the point process of the eigenvalues of random permutation matrices (e.g. the correlation measure of order $q$, the convergence of the normalized and non-normalized empirical spectral distribution,
 etc.).
 This has been shortly sketched out in the pioneering work by Diaconis and Shahshahani \cite{diaconissha}  and further developed by Wieand in \cite{wieand}, who studied the problem of counting how many eigenvalues lie in some fixed arc of the unit circle. Wieand compares the results obtained in this case with those obtained for the unitary group under the Haar measure and notices some similarities but also some differences when it comes to look at more refined aspects. Then it is suggested that one should try to compute finer statistics  related to the eigenvalues in order to see how random permutation matrices fit into the random matrix picture. Of course one expects some drastic differences: for instance the point process associated with the eigenvalues of random permutation matrices should not be determinantal whereas it is determinantal for the unitary group. The goal of the present work is to continue the work initiated by Diaconis and Shahshahani in \cite{diaconissha} and Wieand in \cite{wieand}. Before mentioning more precisely the various directions in which we wish to extend the existing work, it should be mentioned that other works have been recently done on random permutation matrices, such as the paper by Hambly, Keevach, O'Connell and Stark \cite{HKOS} on the characteristic polynomial of random unitary matrices, the papers by Zeindler (\cite{zz}) and by Dehaye and Zeindler (\cite{dzz}), or  the works by Wieand (\cite{wieand2}) and Evans (\cite{evans}) on the eigenvalues of random wreath products (but in this latter case, the techniques that are involved are different and we shall also address this framework in a future work).

We now briefly mention the way we shall continue and extend some of the previous works and shall postpone precise definitions to the next paragraph:
\begin{itemize}
\item We shall consider a larger ensemble of random matrices with more general distributions; roughly speaking, we first pick a permutation of size $N$ from $\Sigma_N$ (the group of permutations of size $N$) at random according to the Ewens measure of parameter $\theta\geq0$ (under this measure, the probability of a given permutation $\sigma$ is proportional to $\theta^n$, where $n$ is the number of cycles of $\sigma$) and then consider the corresponding matrix; we then replace the $1$'s by a sequence of i.i.d. random variables $z_1,\ldots,z_N$ taking values in $\mathbb{C}^*$. This ensemble is a sub-group of the linear sub-group $GL(N,\mathbb{C})$ and the classical random permutation matrices correspond to the case where $\theta=1$ and where the distribution of the $z_i$'s is the Dirac measure at $1$. The choice of the Ewens measure is natural since it is a one parameter deformation of the uniform distribution which is coherent with the projections from $\Sigma_{N^{'}}$ onto $\Sigma_N$ for $N^{'}\geq N$ (see the next paragraphs).  If the $z_i$'s take their values on the unit circle, then our ensemble is a sub-group of $U(N)$.
\item We shall also give a meaning to almost sure convergence for the probability empirical spectral distribution. Indeed, it is not a priori obvious to define almost sure convergence in random matrix theory since the probability space changes with $N$ and to the best of our knowledge there has not been, so far, a satisfactory way to overcome this difficulty.  We propose to deal with this problem by considering the space of \emph{virtual permutations}, which were introduced by Kerov, Olshanski and Vershik in \cite{KOV}. A virtual permutation is a sequence of permutations $(\sigma_N)_{N\geq1}$ constructed in a coherent way, in the sense that for all $N' \geq N \geq 1$, the cycle structure of $\sigma_N$ can be obtained from the cycle structure of $\sigma_{N^{'}}$ simply by removing the elements strictly larger than $N$. 
The virtual permutations satisfy the following property: if $(\sigma_N)_{N \geq 1}$ is a virtual permutation, if $N' \geq N \geq 1$ and if $\sigma_{N^{'}}$ is distributed according to the Ewens measure of parameter $\theta$ on $\Sigma_{N^{'}}$, then $\sigma_N$ follows the Ewens measure of parameter $\theta$ on $\Sigma_N$. We shall then take advantage of the work by Tsilevich \cite{tsi1}, who proved almost sure convergence for the relative cycle lengths of virtual permutations. We shall also extensively use an algorithmic way to generate the Ewens measure on the space of virtual permutations.
\item We shall study in detail the point process of the eigenvalues of a matrix $M$ drawn from our matrix ensembles. For instance we establish various convergence results for the empirical spectral measure $\mu(M)=\sum \delta_\lambda$, or for $\mu(M)/N$ ($N$ being the dimension of the matrix) where the sum is over all eigenvalues $\lambda$ of $M$, counted with multiplicity. We also consider the average measure obtained by taking the expectation of $\mu(M)$, as well as the correlation measure of order $q$. In the special case where the random variables $z_i$'s take their values on the unit circle, the point process associated with the eigenvalues  has some more remarkable properties. For instance, when the $z_i$'s are uniformly distributed on the unit circle,  the empirical spectral measure as well as its limit (in a sense to be made precise) are invariant by translation. Still in this case, the $1$-correlation for fixed $N$ and $N=\infty$ is the Lebesgue measure, whereas for $q\geq2$, and fixed $N$, the $q$-correlation is not absolutely continuous with respect to the Lebesgue measure anymore. For $N=\infty$, the pair correlation measure is still absolutely continuous with respect to the Lebesgue measure with an explicit density, but this result fails to hold for the correlations of order greater than or equal to $3$. One can push further the analogy with the studies made for the classical compact continuous groups by characterizing the distribution function of the smallest eigenangle (again when the $z_i$'s are uniformly distributed on the unit circle) as the solution of some integral equation. 
\end{itemize}

\subsection{Definitions and notation}\label{sec::1.2}

In this section we describe the structure of the eigenvalues of the "generalized permutation matrices" 
mentioned above, independently of any probability measure, and then we define the family of probability measures which will be studied in detail in the sequel of this article.
More precisely let $\Sigma_N$ be the group of permutations of order $N$.
By straightforward computations, one can check that 
 the set of matrices $M$ such that there exists a permutation
 $\sigma \in \Sigma_N$ and complex numbers $z_1, z_2,..., z_N$, different from zero, 
satisfying $M_{jk} = z_j \mathds{1}_{j = \sigma(k)}$, is a multiplicative group, denoted by $\mathcal{G}(N)$
in this article, and which can be written 
as the wreath product of $\mathbb{C}^*$ and $\Sigma_N$. The group $\mathcal{G}_N$ can also be viewed as the group generated by the permutation matrices and  the diagonal matrices of $GL(N,\mathbb{C})$. The elements of $\mathcal{G}(N)$ such that 
$|z_j| = 1$ for all $j$, $1 \leq j \leq N$, form a subgroup noted $\mathcal{H} (N) $ of $\mathcal{G} (N)$, 
which can be viewed as the wreath product of $\mathbb{U}$ (the set of complex numbers with modulus one) and 
$\Sigma_N$. One can also define, for all integers $k \geq 1$, the group $\mathcal{H}_k (N)$ of elements of
$\mathcal{G} (N)$ such that $z_j^k =1$ for $1 \leq j \leq N$: the subgroup $\mathcal{H}_k(N)$ of 
$\mathcal{H}(N)$  is the wreath product of $\mathbb{U}_k$ (group of $k$-th roots of unity) and $\Sigma_N$. 
 Note that the group structure of $\mathcal{G}$, $\mathcal{H}$ or 
$\mathcal{H}_N$ does not play a fundamental role in our work, however, this structure gives a parallel between
this paper and the study of other groups of matrices, such as the orthogonal group, the unitary group or the symplectic group (we 
note that $\mathcal{H}$ and $\mathcal{H}_N$ are subgroups of the unitary group of dimension $N$). 
The advantage of the study of matrices in $\mathcal{G}(N)$ is the fact that the structure of their
 eigenvalues can be very explicitly described in function of the cycle structure of the corresponding permutations.
 More precisely, let $M$ be a matrix satisfying $M_{jk} = z_j \mathds{1}_{j = \sigma(k)}$ for
all $1 \leq j,k \leq N$, where $\sigma$ is a permutation of order $N$, and $z_1,...,z_N \in \mathbb{C}^*$.
If the supports of the cycles of $\sigma$ are $C_1,...,C_n$, with corresponding cardinalities $l_1,...,l_n$,
and if for $1 \leq m \leq n$, $R_m$ is the set of the roots of order $l_m$ of the complex number
$$ Z_m :=\prod_{j \in C_m} z_j,$$ 
then the set of eigenvalues of $M$ is the union of the sets $R_m$, and the multiplicity of any eigenvalue is equal 
to the number of sets $R_m$ containing it. 
An example of explicit calculation is that of the
trace: since for all integers $l \geq 2$, the 
sum of the $l$-th roots of unity is equal to zero, one immediately deduces that
$$\operatorname{Tr}(M) = \sum_{j \in F} z_j,$$
where $F$ is the set of fixed points of $\sigma$. 
More generally, one can compute the trace of all the powers of $M$. Indeed, for all integers $k \geq 1$, the
 eigenvalues of $M^k$ can be computed 
by taking the $k$-th powers of the elements of the sets $R_m$. Therefore, 
$$\operatorname{Tr}(M^k) = \sum_{m = 1}^n \sum_{\omega^{l_m} = Z_m} \omega^k.$$
Now, if $l_m$ is not a divisor of $k$, the last sum is equal to zero, and if $l_m$ is a 
divisor of $k$, all the terms of the last sum are equal to $Z_m^{k/l_m}$. We deduce that
$$\operatorname{Tr}(M^k) = \sum_{l_m | \, k} l_m Z_m^{k/l_m}.$$
We see that the description of the eigenvalues and the computations above 
do not depend on any probability measure given on the space $\mathcal{G}(N)$.
Let us now define a particular class of probability measures on $\mathcal{G}(N)$ which will be studied in detail in this paper. 
\begin{definition}
Let $\theta > 0$ and let $\mathcal{L}$ be a probability distribution on $\mathbb{C}^*$. The probability measure
$\mathbb{P} (N, \theta, \mathcal{L})$ on $\mathcal{G}(N)$ is the law of the matrix $M(\sigma,z_1,...,z_N)$, where:
\begin{itemize}
\item the permutation $\sigma$ follows the Ewens measure of parameter $\theta$ on $\Sigma_N$, i.e. the probability
that $\sigma$ is equal to a given permutation is proportional to $\theta^{n}$, where $n$ is 
the number of cycles of $\sigma$.
\item for all $j$, $1 \leq j \leq N$, $z_j$ is a random variable following the probability law $\mathcal{L}$.
\item the random permutation $\sigma$ and the random variables $z_1,...,z_N$ are all independent. 
\item $M(\sigma,z_1,...,z_N)$ is the matrix $M \in \mathcal{G}(N)$ such that for all $1 \leq j,k \leq N$,
$M_{jk} = z_j \mathds{1}_{j = \sigma(k)}$. 
\end{itemize}
\noindent
 \end{definition}
In this paper we prove for a large class of probability distributions $\mathcal{L}$, the weak convergence of the law
  of the empirical measure of 
the eigenvalues, when the dimension $N$ tends to infinity.
 In several particular cases we are also interested in almost sure convergences, and then we 
 need to couple all the dimensions
$N$ on the same probability space. This can be done by introducing the so-called 
virtual permutations, which were first defined by Kerov, Olshanski and Vershik in \cite{KOV} and also studied by Tsilevich \cite{tsi1}.
A virtual permutation is a sequence $(\sigma_N)_{N \geq 1}$ of permutations, such that 
for all $N \geq 1$, $\sigma_N \in \Sigma_N$, and the cycle structure of $\sigma_N$ is obtained 
from the cycle structure of $\sigma_{N+1}$, simply by removing the element $N+1$ (for example
if $\sigma_8 = (13745)(28)(6)$, then $\sigma_7 = (13745)(2)(6)$). Now, for $\theta > 0$, it is possible
to define on the space of virtual permutations the so-called Ewens measure of parameter $\theta$ 
as the unique (by the monotone class theorem) probability measure, such that if $(\sigma_N)_{N \geq 1}$ 
follows this measure, $\sigma_N$ follows the Ewens($\theta$) measure on $\Sigma_N$. 
Now we can introduce the following definition:
\begin{definition}
Let $\theta > 0$ and let $\mathcal{L}$ be a probability law on $\mathbb{C}^*$. The probability measure
$\mathbb{P} (\infty, \theta, \mathcal{L})$, defined on the product of the probability spaces 
$\mathcal{G}(N)$, $N \geq 1$,
 is the law of a sequence of random matrices $(M_N)_{N \geq 1}$, such that
 $M_N = M (\sigma_N, z_1,...,z_N)$, where:
\begin{itemize}
\item the sequence $(\sigma_N)_{N \geq 1}$ is a random virtual permutation following 
Ewens measure of parameter $\theta$.
\item for all $j \geq 1$, $z_j$ is a random variable following the distribution $\mathcal{L}$.
\item the virtual permutation $(\sigma_N)_{N \geq 1}$ and the random variables $(z_j)_{j \geq 1}$
are independent.
\end{itemize}
\end{definition}
\noindent
It is easy to check that for all $N \geq 1$, the image of $\mathbb{P} (\infty, \theta, \mathcal{L})$
by the $N$-th coordinate projection is the measure $\mathbb{P} (N, \theta, \mathcal{L})$. 
The properties on $\mathbb{P} (N, \theta, \mathcal{L})$
or $\mathbb{P} (\infty, \theta, \mathcal{L})$ which are obtained in this article depend on the
 probability distribution $\mathcal{L}$ in an essential way. 
 
 In the next section, we shall review  some properties of virtual permutations which are needed in the sequel of the paper, and then we give general results on the point process associated with the eigenvalues of random matrices from $\mathcal{G}(N)$. We finally refine some of these results in the case of $\mathcal{H}(N)$.

\section{Generating the Ewens measure on the set of virtual permutations}\label{sec::2}

The space of virtual permutations was first introduced by Kerov, Olshanski and Vershik in \cite{KOV} in the context of representation theory; the interested reader can refer to the notes by Olshanski in \cite{ver} for more details and references. Here we shall  mostly be concerned with the probabilistic aspects of virtual permutations which were studied in detail by Tsilevich in \cite{tsi1}. We now review a construction of the virtual permutations which is explained in \cite{tsi1} and which is suitable for the probabilistic reasoning. We then show how to generate the Ewens measure of parameter $\theta$ on the space of virtual permutations. This is already explained by other means by Tsilevich in her unpublished note \cite{tsi2}; here we provide a more elementary way to generate it and give all details since this is going to be at the source of many of our proofs and because also virtual permutations might not be so well-known. In the sequel, we shall assume that the reader is familiar with the GEM and Poisson-Dirichlet distributions (if not, one can refer to \cite{bill}, p. 40--48).

 As already mentioned, one of the interests of the construction of virtual permutations  is that it gives a natural explanation
of convergences in law involved when one looks at the relative lengths of cycles of random permutations
in the symmetric group $\Sigma_N$ when $N$ goes to infinity. Recall that for all integer $N \geq 1$, 
there exists a bijective map
$$\Phi_N \, : \, \prod_{j=1}^N \{1,2,...,j\} \, \longrightarrow \, \Sigma_N$$
such that 
$$\Phi_N \left( (m_j)_{1 \leq j \leq N} \right) = \tau_{N, m_N} \circ \tau_{N-1,m_{N-1}} \circ \, \dots \, \circ
\tau_{2, m_2} \circ \tau_{1, m_1},$$
where for $\tau_{j,k}$ is the unique permutation such that $\tau_{j,k} (j) = k$, 
$\tau_{j,k} (k) = j$, and $\tau_{j,k} (l) = l$ for $l$ different from $j$ and $k$ (if $j \neq k$, $\tau_{j,k}$
is a transposition, if $j=k$, it is the identity). If $N_1 \geq N_2$, the bijections $\Phi_{N_1}$ and 
$\Phi_{N_2}$ induce a natural surjective map $\pi_{N_1, N_2}$ from $\Sigma_{N_1}$ to $\Sigma_{N_2}$, 
defined in the following way: if $\sigma \in \Sigma_{N_1}$, there exists
 a unique sequence $(m_j)_{1 \leq j \leq N_1}$, 
$m_j \in \{1,...,j\}$ such that 
$$\sigma = \Phi_{N_1} \left( (m_j)_{1 \leq j \leq N_1} \right),$$
and one then defines:
$$\pi_{N_1,N_2} (\sigma) := \Phi_{N_2} \left( (m_j)_{1 \leq j \leq N_2} \right).$$
Note that for $N_1 \geq N_2 \geq N_3$, $$\pi_{N_2, N_3} \circ \pi_{N_1, N_2} = \pi_{N_1, N_3}.$$
Now, a virtual permutation is a sequence of permutations $(\sigma_N)_{N \geq 1}$, such that $\sigma_N \in \Sigma_N$
for all $N$, and which is consistent with respect to $\pi$, i.e. for all $N_1 \geq N_2$, 
$$\pi_{N_1,N_2} (\sigma_{N_1}) = \sigma_{N_2}.$$
We denote by $\Sigma_{\infty}$ the (infinite) set of virtual permutations, and note that the group structure 
of $\Sigma_N$ does not induce any group structure on $\Sigma_{\infty}$. Moreover, there is 
a natural bijection, $$\Phi_{\infty} \, : \, \prod_{j=1}^{\infty} \{1,2,...,j\} \, \longrightarrow \, \Sigma_{\infty}
$$ induced by the bijections $\Phi_N$. Indeed, for all infinite sequences $(m_j)_{j \geq 1}$, $m_j \in 
\{1,...,j\}$ one defines:
$$\Phi_{\infty} \left((m_j)_{j \geq 1} \right) :=  (\sigma_N)_{N \geq 1}$$
where for all $N$,
$$ \sigma_N = \Phi_N \left( (m_j)_{1 \leq j \leq N} \right).$$
Here, one immediately checks that $(\sigma_N)_{N \geq 1}$ is consistent. Moreover, one can also define a
surjective map $\pi_{\infty,N}$ from $\Sigma_{\infty}$ to $\Sigma_{N}$ for all $N \geq 1$, by setting:
$$\pi_{\infty,N} \left((\sigma_n)_{n \geq 1} \right) := \sigma_N,$$
and one can check the relation
$$\pi_{N_1, N_2} \circ \pi_{\infty,N_1} = \pi_{\infty, N_2}.$$
 The cycle structure of a virtual 
permutation can be described by the so-called Chinese restaurant process, described for example
by Pitman (\cite{pitman}). More precisely, let $$\sigma_{\infty} = (\sigma_N)_{N \geq 1} \in \Sigma_{\infty}$$
be a virtual permutation. There exists a unique sequence $(m_j)_{j \geq 1}$, $m_j \in \{1,...,j\}$
such that $$ \sigma_{\infty} = \Phi_{\infty} \left( (m_j)_{j \geq 1} \right).$$ Then the cycle 
structure for $\sigma_N$ can be described by induction on $N$:
\begin{itemize}
\item $\sigma_{1}$ is (of course!) the identity of $\Sigma_1$; 
\item if $m_{N+1} = N+1$ for $N \geq 1$, the cycle structure of $\sigma_{N+1}$ is obtained from 
the structure of $\sigma_N$ by simply adding the fixed point $N+1$;
\item if $m_{N+1} \leq N$ for $N \geq 1$, the cycle structure of $\sigma_{N+1}$ is obtained from 
the structure of $\sigma_N$ by inserting $N+1$ just before $m_{N+1}$, in the cycle containing $m_{N+1}$.
\end{itemize}
\noindent
For example, if $\sigma_4 = (124)(3)$ and $m_5 = 5$, then $\sigma_5 = (124)(3)(5)$, and if 
$\sigma_4 = (124)(3)$ and $m_5 = 2$, then $\sigma_5 = (1524)(3)$. As described in \cite{tsi1}, one can define
on $\Sigma_{\infty}$ the so-called Ewens measure, which is the equivalent on virtual permutations of  the Ewens measure 
on $\Sigma_{N}$. More precisely, let $\theta \in \mathbb{R}_+$ be a parameter, and let $(M_j)_{j \geq 1}$
be a sequence of independent random variables, $M_j \in \{1,...,j\}$ such that for all $j \geq 2$:
$$\mathbb{P} [M_j = j]= \frac{\theta}{\theta + j - 1},$$
and 
$$\mathbb{P} [M_j = k] = \frac{1}{\theta + j - 1}$$
for all $k < j$. On $\Sigma_{\infty}$, the Ewens measure $\mu_{\infty}^{(\theta)}$ of parameter $\theta$
 is defined as the image of the law of $(M_j)_{j \geq 1}$ by the map $\Phi_{\infty}$. The name of this 
measure is consistent, since the image of $\mu_{\infty}^{(\theta)}$ by $\pi_{\infty, N}$ is precisely the
Ewens measure $\mu_{N}^{(\theta)}$ of parameter $\theta$ on $\Sigma_N$, under which the probability of
a permutation $\sigma_N$ is given by the expression: 
$$\frac{\theta^{c (\sigma_N) - 1} }{(\theta + 1)(\theta + 2)...(\theta + N - 1)},$$
where $c  (\sigma_N)$ is the number of cycles of $\sigma_N$.

\textbf{Remark:} For $\theta=1$, the Ewens measure on $S_N$ is the uniform measure, and for $\theta=0$, 
it is the uniform measure on permutations with a unique cycle. 

In \cite{tsi1} it is proved that if $(\sigma_N)_{N \geq 1}$ is a virtual permutation following the
Ewens measure of parameter $\theta$, then for all $k \geq 0$, the $k$-th length of cycle (by decreasing order)
corresponding to the permutation $\sigma_N$, divided by $N$, tends a.s. to a random variable $x_k$ when $N$ goes
to infinity. Moreover, the decreasing sequence $(x_k)_{k \geq 1}$ follows the Poisson-Dirichlet distribution of
parameter $\theta$. This property can in fact be easily  explained by the construction of the
Ewens measure on $\Sigma_{\infty}$ we give below. 

Let $\lambda = (\lambda_j)_{j \geq 1}$ be a decreasing sequence in $\mathbb{R}_+$ and let us denote:
$$K(\lambda):= \inf \{k \geq 1, \lambda_k =  0 \} \in \mathbb{N^*} \cup \{\infty\}.$$
The set $E(\lambda)$ is defined as the disjoint union of circles $(C_j)_{1 \leq j < K(\lambda)}$, such 
that $C_j$ has perimeter $\lambda_j$. Now let $x = (x_k)_{k \geq 1}$ be a sequence of distincts 
points in $E(\lambda)$. One defines a virtual permutation $\sigma_{\infty} (\lambda, x) = \left(\sigma_N(\lambda,x)
\right)_{N \geq 1}$ as 
follows: for $N \geq 1$, $k \in \{1,...,N\}$, there exists a unique $j$ such that the point $x_k$ 
lies on the circle $C_j$. Let us follow the circle $C_j$, counterclockwise, starting from 
$x_k$: the image of $k$ by $\sigma_N(\lambda,x)$ is the index of the first point in $\{x_1,...,x_N\}$
we encounter after $x_k$ (for example, if $x_k$ is the only point in $C_j$ and $\{x_1,...,x_N\}$,
then $k$ is a fixed point of $\sigma_N(\lambda,x)$, because starting from $x_k$, we do a full turn of
the circle $C_j$, before encountering $x_k$ again). 
The cycle structure of $\sigma_N(\lambda, x)$ is the following: two elements $k$ and $l$ in $\{1,...,N\}$ 
are in the same cycle if and only if $x_k$ and $x_l$ lie on the same circle, and the order of the elements $\{k_1,...,k_p\}$
in a given cycle corresponds to the counterclockwise order of the points $x_{k_1},...,x_{k_p}$, which are 
on the same circle. Moreover, the cycle structure of $\sigma_{N+1}(\lambda,x)$ can be obtained from 
the structure of $\sigma_N(\lambda,x)$ by a "chinese restaurant" process:
\begin{itemize}
\item If $x_{N+1}$ is on a circle which does not contain any of the points $x_1,...,x_N$, then 
one simply adds the fixed point $N+1$;
\item If $x_{N+1}$ is on a circle which contains some of the points $x_1,...,x_N$, and if $x_{N+1}$
lies just before $x_p$ if one follows this circle counterclockwise, then one inserts $N+1$ in 
the cycle containing $p$, just before $p$.
\end{itemize}
\noindent
This construction implies that $\left( \sigma_N(\lambda,x) \right)_{N \geq 1}$ is a consistent sequence of 
permutations, and then $\sigma_{\infty}( \lambda,x)$ is a virtual permutation. Therefore a virtual permutation
can be viewed as a chinese restaurant process with continuous tables (the circles $(C_j)_{1 \leq j < K(\lambda)}$), and
 with an infinite number 
of customers (the points $(x_k)_{k \geq 1}$): its component of index $N$ is obtained by taking into 
account only the $N$ first customers. 
 Note that for the moment, $\lambda$, 
the sequence of lengths of the circles (or the tables!) plays a minor role in our construction. However, it
becomes important when one introduces randomness. More precisely, in this framework,
one obtains the following construction
of the Ewens measure on $\Sigma_{\infty}$ (and therefore, on $\Sigma_N$, by using $\pi_{\infty, N}$):
\begin{proposition} \label{ewens}
Let $\theta \in \mathbb{R}_+$. If the random sequence $\lambda$ follows the Poisson-Dirichlet distribution of parameter 
$\theta$ (for $\theta = 0$, one sets $\lambda_1 = 1$ and $\lambda_k = 0$ for $k \geq 1$), and if, 
conditionally on $\lambda$, the points $(x_k)_{k \geq 1}$ are i.i.d., with distribution absolutely continuous with respect to the Lebesgue measure (e.g. the uniform distribution\footnote{Here, uniform 
on $E(\lambda)$ means
the following: $x_k$ lies in $C_j$ with probability $\lambda_j$, the perimeter of $C_j$ (note that 
the sum of the perimeters is a.s. equal to one), and conditionally on $x_k \in C_j$, $x_k$ is 
uniform on the circle $C_j$.
})
on $E(\lambda)$ (and hence, a.s. distinct), then for $x := (x_k)_{k \geq 1}$, the virtual permutation
 $\sigma_{\infty} \left(\lambda,x  \right)$ follows the Ewens measure of parameter $\theta$. 
 \end{proposition}
\begin{proof} The probability law of a random virtual permutation
is uniquely determined by its image by $\Phi_{\infty}^{-1}$, which is the probability law 
of a sequence $(M_j)_{j \geq 1}$ of random variables ($M_j \in \{1,...,j\}$). By the monotone class 
theorem, this law is uniquely determined by the sequence of laws of $(M_j)_{1 \leq j \leq N}$,
$N \geq 1$. Now, by applying $\Phi_N$ for all $N \geq 1$, one deduces that the law of a virtual 
permutation is uniquely determined by the sequence of laws of its images by $\pi_{\infty,N}$,
$N \geq 1$. This property implies Proposition \ref{ewens} if one shows that for
all $N \geq 1$, $\sigma_{N} \left(\lambda,x \right)$ follows the Ewens measure of parameter $\theta$. To prove this, let 
us first observe that for all permutations $\psi \in S_{N}$, the law of $(x_{\psi(k)})_{1 \leq k \leq N}$
is equal to the law of $(x_k)_{1 \leq k \leq N}$. Hence, the law of $\sigma_{N} \left(\lambda,
x \right)$ is invariant by conjugation. This implies that the probability of 
a given permutation, under this law, depends only on its cycle structure, as under the Ewens measure.
Therefore, it is sufficient to prove that for all partitions $(l_1,...,l_p)$ of $N$, the 
probability that the supports of the cycles of $\sigma_{N} \left(\lambda,x \right)$ are
exactly the sets of the form $$\{l_1 + ... + l_n + 1, l_1 + ... + l_n + 2, \dots, l_1 + l_2 + ...+ l_{n} + l_{n+1} \}$$
for $0 \leq n \leq p-1$, is the same as under the  Ewens measure. Now, conditionally 
on $\lambda$, this probability can be written as follows:
$$\sum_{i_1 \neq i_2 \neq... \neq i_p} \prod_{n=1}^p \lambda_{i_n}^{l_n}.$$
Hence one only needs to prove the equality:
\begin{equation}
\mathbb{E} \left[\sum_{i_1 \neq i_2 \neq... \, \neq i_p} \prod_{n=1}^p \lambda_{i_n}^{l_n} \, \right] =
\frac{\theta^{p-1}}{(\theta+1)...(\theta + N-1)} \, \prod_{n=1}^p  (l_n - 1)! \label{1}
\end{equation}
since its right-hand side is the probability of the event described above, under the Ewens measure of parameter 
$\theta$. For $\theta = 0$, both sides are equal to one if $p=1$ and $l_1 = N$, and to zero
otherwise: therefore one can assume that $\theta > 0$. Let $(\mu_j)_{j \geq 1}$ be a random sequence 
following the GEM law of parameter $\theta$, and let $(\nu_j)_{j \geq 1}$ be the sequence obtained 
by putting $(\mu_j)_{j \geq 1}$ in decreasing order. One  obviously has
$$\sum_{i_1 \neq i_2 \neq \dots \neq i_p} \prod_{n=1}^p \mu_{i_n}^{l_n} = 
 \sum_{i_1 \neq i_2 \neq \dots \neq i_p} \prod_{n=1}^p \nu_{i_n}^{l_n}.$$
Now, $(\nu_j)_{j \geq 1}$ and $(\lambda_j)_{j \geq 1}$ have the same law, hence it is sufficient to prove that
\begin{equation}
\mathbb{E} \left[\sum_{i_1 \neq i_2 \neq \dots \neq i_p} \prod_{n=1}^p \mu_{i_n}^{l_n} \, \right] =
\frac{\theta^{p-1}}{(\theta+1)...(\theta + N-1)} \, \prod_{n=1}^p  (l_n - 1)! \label{nh28}
\end{equation}
One needs the following lemma:
\begin{lemma} \label{a}
Let $(\mu_j)_{j \geq 1}$ be a GEM process of parameter $\theta>0$, and $r, s \in \mathbb{R}_+$. 
Then the quantity
$$E(\theta,r,s) := \mathbb{E} \left[ \sum_{i=1}^{\infty} \mu_i^r \left( 1 - \sum_{j=1}^{i} \mu_j \right)^s  \, \right]$$
satisfies the equality
$$E(\theta,r,s) = \frac{r! (s+ \theta-1)! \theta}{(r+s+ \theta - 1)! (r+s)}.$$
\end{lemma}
\noindent
\textbf{Proof:} One can write
\begin{align*} 
E(\theta,r,s) & = \mathbb{E} \left[ \mu_1^r (1-\mu_1)^s \right] \\ & + 
\mathbb{E} \left[ (1- \mu_1)^{r+s} \, \mathbb{E} \left[ \sum_{i=2}^{\infty} \,
 \left( \left. \frac{\mu_i}{1 - \mu_1} \right)^r \left( 1 - \sum_{j=2}^{i} \frac{\mu_j}{1 - \mu_1} \right)^s \, 
\right| \mu_1 \right] \, \right].
\end{align*}
\noindent
Now conditionally on $\mu_1$, $\left( \frac{\mu_{j+1}}{1 - \mu_1} \right)_{j \geq 1}$ is a GEM process of 
parameter $\theta$. Therefore:
\begin{align*}
\mathbb{E} \left[ \sum_{i=2}^{\infty} \,
 \left( \left. \frac{\mu_i}{1 - \mu_1} \right)^r \left( 1 - \sum_{j=2}^{i} \frac{\mu_j}{1 - \mu_1} \right)^s \, 
\right| \mu_1 \right]  & = \mathbb{E} \left[ \sum_{i=1}^{\infty} \mu_i^r \left( 1 - \sum_{j=1}^{i} \mu_j
 \right)^s \right] \\ & = E(\theta,r,s).
\end{align*}
which implies
$$  E(\theta,r,s) = \mathbb{E} [ \mu_1^r (1 - \mu_1)^s]  + E(\theta,r,s) \, \mathbb{E} [(1 - \mu_1)^{r+s}],$$
and, since the density of the law of $\mu_1$, with respect to the Lebesgue measure is $\theta \, (1-x)^{\theta-1}$
on $(0,1)$,
\begin{align*}
 E(\theta,r,s) & = \frac{\mathbb{E} [ \mu_1^r (1 - \mu_1)^s] }{1 -  \mathbb{E} [(1 - \mu_1)^{r+s}]}
= \frac{ \theta \, \int_0^1 x^r (1-x)^{s + \theta - 1} dx}{ 1 - \theta \int_0^1 (1-x)^{r+s+ \theta - 1} dx}
\\ & = \frac{\theta \, r! (s+ \theta - 1)!/ (r+s+ \theta)!} {1 - \theta/(r+s+ \theta)} = 
\frac{\theta \, r! (s+ \theta - 1)! (r+s+ \theta)} {(r+s)  (r+s+ \theta)!} 
\end{align*}
which implies Lemma \ref{a}. 

Now let us go back to the proof of Proposition \ref{ewens}. For all integers $q \geq 0$, $1 \leq j_1 < ...
 \, <j_q$, and $r_1,...,r_q,r,s \geq 0$:
\begin{align}
& \; \mathbb{E} \left[ \left. \sum_{j = j_{q}+1}^{\infty} \left( \prod_{p=1}^{q} \mu_{j_p}^{r_p} \right) 
\, \mu_j^r \left(1 - \sum_{i=1}^j \mu_i \right)^s \, \right| \, (\mu_i)_{1 \leq i \leq j_q} \right] 
=  \left( \prod_{p=1}^{q} \mu_{j_p}^{r_p} \right)  \left(1 - \sum_{i=1}^{j_q} \mu_i \right)^{r+s} \nonumber \\ & 
\times \mathbb{E}\left[ \left. \sum_{j = j_{q}+1}^{\infty} \, \left( \frac{\mu_j} {1 -
 \sum_{i=1}^{j_q} \mu_i } \right)^r 
\left( 1 - \sum_{i=j_q + 1}^{j} \frac{\mu_i}{1 -  \sum_{i=1}^{j_q} \mu_i } \right)^s \right| 
\, (\mu_i)_{1 \leq i \leq j_q} \right]. \nonumber
\end{align}
\noindent
Since conditionally on $(\mu_i)_{1 \leq i \leq j_p}$, $$\left( \frac{\mu_{j_q + j'}}{1 - \sum_{i = 1}^{j_q} 
\mu_i } \right)_{j' \geq 1}.$$ is a GEM process of parameter $\theta$, the last conditional expectation
 is equal to $E(\theta,r,s)$. One deduces that
\begin{align*}
& \; \mathbb{E} \left[ \sum_{j = j_{q}+1}^{\infty} \left( \prod_{p=1}^{q} \mu_{j_p}^{r_p} \right) 
\, \mu_j^r \left(1 - \sum_{i=1}^j \mu_i \right)^s \right] \\ & = E(\theta,r,s) \, \mathbb{E} \left[ 
 \left( \prod_{p=1}^{q} \mu_{j_p}^{r_p} \right) 
\left(1 - \sum_{i=1}^{j_q} \mu_i \right)^{r+s} \right] .
\end{align*}
\noindent
By considering all the possible values of $(j_p)_{1 \leq p \leq q}$, and by adding the equalities, one
obtains that
\begin{align*}
& \; \mathbb{E} \left[ \sum_{j_1<j_2<\,...<j_{q+1}} \left( \prod_{p=1}^{q} \mu_{j_p}^{r_p} \right) 
\, \mu_{j_{q+1}}^r \left(1 - \sum_{i=1}^{j_{q+1}} \mu_i \right)^s \right] \\ & = E(\theta,r,s) \, \mathbb{E}
 \left[ \sum_{j_1<j_2<\,...<j_{q}}  \left( \prod_{p=1}^{q} \mu_{j_p}^{r_p} \right) 
\left(1 - \sum_{i=1}^{j_q} \mu_i \right)^{r+s} \right] .
\end{align*}
\noindent
By applying recursively this equality, and using Lemma \ref{a}, one deduces that
\begin{align*}
\mathbb{E} \left[ \sum_{j_1<j_2<\,...<j_{q}} \left( \prod_{p=1}^{q} \mu_{j_p}^{r_p} \right) \right]
 & = \prod_{p=1}^q \left[ E \left(\theta,r_p, \sum_{m= p+1}^q r_m \right) \right] \\ & = \prod_{p=1}^q 
\, \frac{\theta \, (r_p)! \, \left[ \left(  \sum_{m= p+1}^q r_m \right) + \theta - 1 \right]!}
{ \left(  \sum_{m= p}^q r_m \right) \, \left[ \left(  \sum_{m= p}^q r_m \right) + \theta - 1 \right]!}
\\ & = \theta^q \, \left( \prod_{p=1}^q r_p! \right) \, \frac{(\theta-1)!}
 { \left[ \left(  \sum_{m= 1}^q r_m \right) + \theta - 1 \right]!} \, \prod_{p=1}^q \frac{1}{\sum_{m=p}^q r_m}.
\end{align*}
We can now compute the left-hand side of \eqref{nh28}:
\begin{align*}
\mathbb{E} \left[\sum_{i_1 \neq i_2 \neq \dots \neq i_p} \prod_{n=1}^p \mu_{i_n}^{l_n} \, \right] 
& = \sum_{\sigma \in \Sigma_p} \mathbb{E} \left[ \sum_{i_{\sigma(1)}<\,...<i_{\sigma(p)}} 
\prod_{n=1}^p \mu_{i_{\sigma(n)}}^{l_{\sigma(n)}} \right]  \\ & = 
 \sum_{\sigma \in \Sigma_p} \mathbb{E} \left[ \sum_{i_1<\,...<i_p} 
\prod_{n=1}^p \mu_{i_n}^{l_{\sigma(n)}} \right]  \\ \; = 
\sum_{\sigma \in \Sigma_p} \theta^p  \left( \prod_{n=1}^p l_{\sigma(n)}! \right) \, &  \frac{(\theta - 1)!}
{ \left[ \left(  \sum_{n= 1}^p l_{\sigma(n)} \right) + \theta - 1 \right]!} \, \prod_{n=1}^p
 \frac{1}{\sum_{m=n}^p l_
{\sigma(m)}} \\ & = \theta^p  \left( \prod_{n=1}^p l_{n}! \right) \, \frac{(\theta - 1)!}{(\theta + N-1)!}
\, \sum_{\sigma \in \Sigma_p} \frac{1}{ \prod_{n=1}^p \sum_{m=n}^p l_
{\sigma(m)}}. 
\end{align*}
\noindent
Therefore, \eqref{nh28}  and then Proposition \ref{ewens}, is proved if one checks the equality:
$$ \sum_{\sigma \in \Sigma_p} \frac{1}{ \prod_{n=1}^p \sum_{m=n}^p l_
{\sigma(m)}} = \frac{1}{\prod_{n=1}^p l_n}. $$
\noindent
Now, since for all $l > 0$,
$$\frac{1}{l}  = \int_{0}^{\infty} e^{-lx} dx,$$
one deduces that
\begin{align*}
 \sum_{\sigma \in \Sigma_p} \frac{1}{ \prod_{n=1}^p \sum_{m=n}^p l_
{\sigma(m)}}  & = \sum_{\sigma \in \Sigma_p}  \int_{x_1,...,x_p \geq 0} 
e^{- \sum_{n=1}^p x_n \left( \sum_{m=n}^p l_{\sigma(m)} \right) } \prod_{n=1}^p dx_n \\ & = 
\sum_{\sigma \in \Sigma_p}  \int_{x_1,...,x_p \geq 0} 
e^{- \sum_{m=1}^p l_{\sigma(m)} \sum_{n=1}^m x_n} \prod_{n=1}^p dx_n.
\end{align*}
\noindent
By doing the change of variable
$$y_{\sigma(m)} = \sum_{n=1}^m x_n,$$
one obtains:
\begin{align*}
 \sum_{\sigma \in \Sigma_p} \frac{1}{ \prod_{n=1}^p \sum_{m=n}^p l_
{\sigma(m)}}  & = \sum_{\sigma \in \Sigma_p}  \int_{y_1,...,y_p \geq 0} 
e^{ - \sum_{m=1}^p l_{\sigma(m)} y_{\sigma(m)} } \mathds{1}_{y_{\sigma(1)} \leq \, ... \,\leq y_{\sigma(p)}}
 \prod_{n=1}^p dy_n \\ & = \int_{y_1,...,y_p \geq 0} 
e^{ - \sum_{m=1}^p l_{m} y_{m} } \prod_{n=1}^p dy_n   = \frac{1}{\prod_{n=1}^p l_n} 
\end{align*}
\noindent
which completes the proof of Proposition \ref{ewens}. 
\end{proof}
As an illustration, we quickly show how this proposition implies the following almost sure convergence result for relative cycle lengths due to Tsilevich:
\begin{proposition}[Tsilevich \cite{tsi1}] \label{lengths}
Let $(\sigma_N)_{N \geq 1}$ be a virtual permutation following the Ewens probability measure with parameter $\theta$.
One defines the sequence $(\alpha_k^{(N)})_{k \geq 1}$ of normalized lengths of cycles of $\sigma_N$ (i.e. 
lengths divided by $N$), ordered by increasing smallest elements, and completed by zeros. Then, for all 
$k \geq 1$, $\alpha_k^{(N)}$ converges almost surely to a random variable $\alpha_k^{(\infty)}$, 
and $(\alpha_k^{(\infty)})_{k \geq 1}$ follows a GEM distribution of parameter $\theta$. In particular, 
the law of $(\alpha_k^{(N)})_{k \geq 1}$ converges weakly to the GEM($\theta$) distribution. Moreover, if $y_l^{(N)}$ 
denotes the $l$-th largest element of $(\alpha_k^{(N)})_{k \geq 1}$ for all integers $N \geq 1$ and $N= \infty$, 
then $$y_l^{(N)} \underset{N \rightarrow \infty}{\longrightarrow} y_l^{(\infty)}$$ a.s., and 
 $(y_l^{(\infty)})_{l \geq 1}$ follows a Poisson-Dirichlet distribution of parameter $\theta$. In particular, 
the law of $(y_l^{(N)})_{l \geq 1}$ (i.e. the sequence of decreasing normalized lengths of cycles) tends 
to the PD($\theta$) distribution. 
\end{proposition}
\noindent
\begin{proof} Let us construct $(\sigma_N)_{N \geq 1}$ via Proposition \ref{ewens}. Since Proposition
\ref{lengths} is trivial for $\theta = 0$, one can suppose $\theta > 0$. This implies (with the notation of 
Proposition \ref{ewens}) that $\lambda_j > 0$ for all $j \geq 1$, and a.s.,
 there exists $p$ such that $x_p$ lies on the circle $C_j$. We define a sequence $(j_n)_{n \geq 1}$
by the following recursive construction:
\begin{itemize}
\item The index $j_1$ is given by: $x_1 \in C_{j_1}$: 
\item For $n \geq 1$, $j_1,...,j_n$ already defined, $j_{n+1}$ is given by: $x_p \in C_{j_{n+1}}$, where $p$
is the smallest index such that $x_p \notin C_{j_1} \cup \, ... \, \cup C_{j_n}$ (this index a.s. exists).
\end{itemize}
\noindent
It is easy to check that for $k \geq 1$:
$$\alpha_k^{(N)} = \frac{ \left| C_{j_k} \cap \{x_1,...,x_N\} \right| }{N} = \frac{1}{N} \, \sum_{p=1}^N
\mathds{1}_{x_p \in C_{j_k}}.$$
Now, by the law of large numbers, it is almost sure that for all integers $j \geq 1$:
$$ \frac{1}{N} \, \sum_{p=1}^N \mathds{1}_{x_p \in C_{j}} \underset{N \rightarrow\infty}{\longrightarrow} 
\lambda_j.$$
Then, $\alpha_k^{(N)}$ tends almost surely to $\lambda_{j_k}$. Now, by construction of the sequence $(j_k)
_{k \geq 1}$, we see that $(\lambda_{j_k})_{k \geq 1}$ is the classical size-biased reordering 
of $(\lambda_j)_{j\geq 1}$, and hence a GEM process of parameter $\theta$. Now, it is obvious 
that $(y_l^{(\infty)})_{l \geq 1}$ is the decreasing reordering of $(\lambda_{j_k})_{k \geq 1}$, i.e. 
the PD($\theta$) process $(\lambda_l)_{l \geq 1}$. It remains to prove that $y_l^{(N)} \rightarrow \lambda_l$
almost surely. Indeed, if $l$ is fixed, there exists a.s. a (random) index $N_0 > l$ such that:
$$\sum_{k=1}^{N_0} \alpha_k^{(\infty)} > 1 - \lambda_l/2.$$
Since $\alpha_k^{(N)} \underset{N \rightarrow \infty}{\longrightarrow}
 \alpha_k^{(\infty)}$ for all $k \geq 1$, there exists a.s. a (random) $N_1$ such that for $N \geq N_1$:
$$\sum_{k=1}^{N_0} \alpha_k^{(N)} > 1 - \lambda_l/2.$$
Since the numbers $(\alpha_k^{(\infty)})_{k \geq 1}$ are a.s. pairwise distinct, there exists a.s. 
$N_2$ such that if $N \geq N_2$, the order of $(\alpha_k^{(N)})_{1 \leq k \leq N_0}$ is the same as
the order of $(\alpha_k^{(\infty)})_{1 \leq k \leq N_0}$. In particular (recall that $N_0 > l$), the 
$l$-th largest element of $(\alpha_k^{(N)})_{1 \leq k \leq N_0}$ has the same index $r$ as the $l$-th 
largest element of $(\alpha_k^{(\infty)})_{1 \leq k \leq N_0}$. Now, since $\alpha_k^{(\infty)} < \lambda_l/2$
for all $k > N_0$, $\alpha_r^{(\infty)}$ is also the $l$-th largest element of  
$(\alpha_k^{(\infty)})_{k \geq 1}$, i.e. $\lambda_l$. For $N \geq \sup (N_1,N_2)$, $\alpha_r^{(N)}$ is 
the $l$-th largest element of $(\alpha_k^{(N)})_{1 \leq k \leq N_0}$ and $\alpha_k^{(N)} < \lambda_l/2$
for $N > N_0$, hence the $l$-th largest element $y_l^{(N)}$ of $(\alpha_k^{(N)})_{k \geq 1}$ 
is included in the interval $[\alpha_r^{(N)}, \alpha_r^{(N)} \vee \lambda_l/2]$. Since $\alpha_r^{(N)}
\underset{N \rightarrow \infty}{\longrightarrow} \alpha_r^{(\infty)} = \lambda_l$, Proposition \ref{lengths}
is proved.
\end{proof}

\section{The "non-unitary case"}
\subsection{The normalized and non-normalized empirical eigenvalues distributions}
Let $M$ be a matrix in $\mathcal{G} (N)$ for some $N \geq 1$. We associate with the point process of the eigenvalues of $M$ 
the finite measure $\mu(M)$ on $\mathbb{C}$ defined by
$$\mu(M) := \sum_{\lambda \in E(M)} m_M(\lambda) \, \delta_{\lambda},$$
where $E(M)$ is the set of eigenvalues of $M$, $m_M(\lambda)$ is the multiplicity of $\lambda$ as 
an eigenvalue of $M$, and $\delta_{\lambda}$ is Dirac measure at $\lambda$. 
By the general description of eigenvalues given in Section \ref{sec::1.2}, one has, for all
$\sigma \in \Sigma_N$, $z_1,...,z_N \in \mathbb{C}^*$:
$$\mu(M(\sigma, z_1,...,z_N)) = \sum_{m=1}^n \sum_{\omega^{l_m}= Z_m} \delta_{\omega},$$
where $l_1,l_2,...,l_n$ are the lengths of the cycles $C_1,C_2...,C_n$ of $\sigma$, and for $1 \leq m \leq n$:
$$Z_m = \prod_{j \in C_m} z_j.$$
Let us now suppose that the distribution of a sequence of random matrices $(M_N)_{N \geq 1}$,
$M_N \in \mathcal{G}(N)$, is of
the form $\mathbb{P}(\infty, \theta, \mathcal{L})$. One has $M_N = M( \sigma_N, z_1,...,z_N)$
where $(\sigma_N)_{N \geq 1}$ follows the Ewens($\theta$) distribution, and is independent of 
the sequence $(z_j)_{j \geq 1}$ of i.i.d. variables, which have law $\mathcal{L}$. 
Since for all $N \geq 1$, the cycle structure of $\sigma_{N}$ can be deduced from the 
cycle structure of $\sigma_{N+1}$ by removing $N+1$, there exists a partition $\Pi$ of $\mathbb{N}^*$
such that for all $N$, the supports of the cycles of $\sigma_N$ are obtained by intersecting the sets of 
$\Pi$ with $\{1,...N\}$. Moreover, under the Ewens($\theta$) measure, $\Pi$
contains a.s. an infinite number of sets (see Section \ref{sec::2}): let us 
order them by increasing smallest elements, and denote them by $(C_m)_{m \geq 1}$.
One then has
\begin{equation}
\mu(M_N) = \sum_{m=1}^{\infty} \, \mathds{1}_{l_{N,m} > 0} \sum_{\omega^{l_{N,m}} = Z_{N,m}}
 \delta_{\omega}, \label{muMN}
\end{equation}
where $l_{N,m}$ is the cardinality of $C_{N,m}$, the intersection  of $C_m$ 
and $\{1,...,N\}$, and $$Z_{N,m} = \prod_{j \in C_{N,m}} z_j.$$
The natural question one can now ask is the behaviour of the measure $\mu(M_N)$ for large $N$. 
Since each cycle of $\sigma_N$ gives a number of eigenvalues equal to its length, one can expect that 
$\mu(M_N)$ is dominated by the large cycles of $M_N$. Moreover, the $l$
eigenvalues corresponding to a cycle of length $l$ form a regular polygon of order $l$, and the 
distance of their vertices to the origin is equal the the $l$-th root of the product of $l$ i.i.d. random 
variables of law $\mathcal{L}$. If $l$ is large and if one can apply a multiplicative version of the 
law of large numbers, one can expect that this distance does not vary too much.
Then, it is natural to guess that under some well-chosen conditions on $\mathcal{L}$, the measure $\mu(M_N)$,
which has total mass $N$, is close to $N$ times the uniform measure on a circle centered at the origin.
Indeed, we can prove the following statement:
\begin{proposition} \label{empirical}
Let $(M_N)_{N \geq 1}$ be a sequence of matrices following the law $\mathbb{P}(\infty, \theta, \mathcal{L})$
for some $\theta > 0$ and some probability $\mathcal{L}$ on $\mathbb{C}^*$. We suppose that if $Z$ is a random 
variable which follows the distribution $\mathcal{L}$, then $\log(|Z|)$ is integrable. Under these assumptions, 
almost surely, the probability measure $\mu(M_N)/N$ converges weakly to the uniform distribution on 
the circle of center zero and radius $\exp \left( \mathbb{E}[\log (|Z|)] \right)$.
\end{proposition}
\begin{proof}
Let $f$ be a continuous and bounded function from $\mathbb{C}$ to $\mathbb{R}$, and let 
$R > 0$. Then, there exists a constant $A > 0$, and a function $\alpha$ from $(0,R)$ to $\mathbb{R}_+$, tending
 to zero at zero, such that for all $\epsilon \in (0,R)$, for all integers $l \geq 1$,
and for all $z \in \mathbb{C}$ such that $|z|^{1/l} \in (R- \epsilon, R + \epsilon)$:
\begin{equation}
\left| \sum_{\omega^l = z} f(\omega) - \frac{l}{2\pi}  \, \int_0^{2 \pi} f(R \, e^{i \lambda}) d \lambda \right| 
\leq \frac{A}{\epsilon} + l \alpha(\epsilon). \label{1}
\end{equation}
Indeed, let us define for all $\delta > 0$: 
$$\eta( \delta) := \sup \{ |f(y)-f(y')|, |y-y'| \leq \delta, |y|, |y'| \leq 2R \},$$
which tends to zero with $\delta$ since $f$ is uniformly continuous on any compact set. 
With this definition, we obtain:
$$\left| \sum_{\omega^l = z} f(\omega) - \sum_{\omega^l = z'} f(\omega) \right| \leq l \eta(\epsilon),$$ 
where $z' := z R^l/|z|$ has modulus $R^l$. Now, there exists $\lambda \in [0, 2 \pi/l)$ such 
that:
$$\sum_{\omega^l = z'} f(\omega) = \sum_{\omega^l = 1} f(R \omega \, e^{i \lambda}) =: \Phi(\lambda)$$
One has, for all $\lambda, \lambda' \in [0, 2\pi/l)$:
$$|\Phi(\lambda) - \Phi(\lambda')| \leq l \eta( R |\lambda - \lambda'|) \leq l \eta( 2 \pi R/l).$$
Moreover
$$\int_0^{2\pi/l} \Phi (\lambda) d \lambda= \int_0^{2\pi} f(R \, e^{i \lambda}) d \lambda,$$
and then, for all $\lambda \in [0, 2\pi/l)$,
$$ \left| \Phi(\lambda) - \frac{l}{2 \pi} \, \int_0^{2 \pi} f(R \, e^{i \lambda}) d \lambda\right| 
\leq l \eta( 2 \pi R/l),$$
which implies:
$$\left| \sum_{\omega^l = z} f(\omega) - \frac{l}{2\pi}  \, \int_0^{2 \pi} f(R \, e^{i \lambda})d \lambda \right| 
\leq l \left[ \eta(\epsilon) + \eta (2 \pi R/l) \right].$$
If $l \leq 2 \pi R/ \epsilon$, one can majorize this quantity by $ 4 \pi R \eta(2 \pi R) / \epsilon $, and
 if $l \geq 2 \pi R/ \epsilon$, one can majorize it by $2l \eta (\epsilon)$. Hence we obtain \eqref{1}. 
 
Since for $B> 0$ depending only on $f$, the left-hand side of \eqref{1} can be trivially majorized by $Bl$
 for any $z \in \mathbb{C}$, we deduce, for $\mu$ equal to the uniform measure on the circle of radius $R$:
$$ \left|\frac{1}{N} \int_{\mathbb{C}} f \, d \mu(M_N) - \int_{\mathbb{C}} f \, d \mu  \right|
\leq \sum_{m=1}^{\infty} \, \mathds{1}_{l_{N,m} > 0}
 \left[ \frac{B \, l_{N,m}}{N} \, \mathds{1}_{|Z_{N,m}|^{1/l_{N,m}} \notin (R-\epsilon, R+ \epsilon)}
+ \frac{A}{N \epsilon} + \frac{l_{N,m}}{N} \, \alpha(\epsilon)\right]$$
For now, let us take $R:= \exp \left( \mathbb{E}[\log (|Z|)] \right)$, where $Z$ is a random variable 
following the law $\mathcal{L}$. By the strong law of large numbers applied to the sequence $(\log |z_j|)_{j \in
C_{m}}$, it is not difficult to check that a.s., for all $m \geq 1$:
$$\frac{B \, l_{N,m}}{N} \, \mathds{1}_{|Z_{N,m}|^{1/l_{N,m}} \notin (R-\epsilon, R+ \epsilon)}$$
tends to zero when $N$ goes to infinity. Moreover, independently of $N$, this quantity is
dominated by $B s_m$, where $s_m$ is the supremum of $l_{N,m}/N$ for $N \geq 1$. For the moment, 
let us assume that a.s.:
\begin{equation}
\sum_{m =1}^{\infty} s_m < \infty \label{2}
\end{equation}
\noindent
In this case, one can apply dominated convergence and obtain:
$$ \sum_{m=1}^{\infty} \, \mathds{1}_{l_{N,m} > 0, |Z_{N,m}|^{1/l_{N,m}} \notin (R-\epsilon, R+ \epsilon)}
  \frac{B \, l_{N,m}}{N} \underset{N \rightarrow \infty}{\longrightarrow} 0.$$
Moreover,
$$\sum_{m=1}^{\infty} \, \mathds{1}_{l_{N,m} > 0} \frac{A}{N \epsilon}
 \underset{N \rightarrow \infty}{\longrightarrow} 0$$
a.s., since the number of cycles of $\sigma_N$ increases slowly with respect to $N$ (the order of magnitude is $\log(N)$),
and 
$$\sum_{m=1}^{\infty} \, \mathds{1}_{l_{N,m} > 0} \frac{l_{N,m}}{N} \, \alpha(\epsilon) = \alpha(\epsilon).$$
 Hence we deduce that
$$ \underset{N \rightarrow \infty}{\lim \sup}
\left|\frac{1}{N} \int_{\mathbb{C}} f \, d \mu(M_N) - \int_{\mathbb{C}} f \, d \mu  \right|
\leq \alpha(\epsilon),$$
and by taking $\epsilon \rightarrow 0$, we are done. It only remains to prove 
\eqref{2}. This relation can be shown by looking carefully at the construction of the Ewens($\theta$) measure
on virtual permutations given in Section \ref{sec::2}. Indeed, for $N \geq 1$, conditionally on 
$(l_{K,m})_{m \geq 1, 1 \leq K \leq N}$, with
 $m_0 := \inf \{m \geq 1, l_{N,m} = 0 \}$, one has $l_{N+1,m} =
 l_{N,m} + \mathds{1}_{m=m_1}$ where $m_1$ is a random index equal to $m'$ with
probability $l_{N,m} / (N+ \theta)$ for $m' < m_0$ and to $m_0$ with
probability $\theta/ (N+ \theta)$. This implies quite easily that $( l_{N,m} /(N+ \theta))_{N \geq 1}$
is a nonnegative submartingale. 
By Doob's inequality one deduces that the expectation of $s_m^2$, is dominated 
by a constant (depending only on $\theta$) times the expectation of $y_m^2$, where $y_m$
 is the limit of $l_{N,m}/N$, which exists almost surely.
Now, $(y_m)_{m \geq 0}$ is a GEM process of parameter $\theta$,
hence the expectation of $y_m^2$ decreases exponentially with $m$. Consequently,
\begin{equation}
\sum_{m = 1}^{\infty} \mathbb{E}[s_m] < \infty, \label{sm}
\end{equation}
\noindent
which implies \eqref{2} almost surely. 
\end{proof}
\noindent
If we do not want to deal with virtual permutations, we can replace the a.s. convergence by a weak convergence
in probability, as follows:
\begin{corollary} \label{empirical2}
Let $(M_N)_{N \geq 1}$ be a sequence of matrices such that $M_N \in \mathcal{G}(N)$ follows the distribution
 $\mathbb{P}(N, \theta, \mathcal{L})$
for some $\theta > 0$ and some probability $\mathcal{L}$ on $\mathbb{C}^*$. We suppose that if $Z$ is a random 
variable which follows the distribution $\mathcal{L}$, then $\log(|Z|)$ is integrable. Under these assumptions, 
the probability measure $\mu(M_N)/N$ converges weakly in probability to the uniform distribution on 
the circle of center zero and radius $R := \exp \left( \mathbb{E}[\log (|Z|)] \right)$, i.e. for all continuous, 
bounded functions $f$ from $\mathbb{C}$ to $\mathbb{R}$:
$$\frac{1}{N} \, \int_{\mathbb{C}} f \, d \mu(M_N) \underset{N \rightarrow \infty}{\longrightarrow} 
\frac{1}{2 \pi} \, \int_0^{2 \pi} f(Re^{i \lambda}) \, d \lambda $$
in probability. 
\end{corollary}
This convergence result means that most
 of the eigenvalues of a matrix $M_N$ following $\mathbb{P} (N, \theta, \mathcal{L})$ 
are concentrated around the circle of radius $\exp \left( \mathbb{E}[\log (|Z|)] \right)$.
Now, even for $N$ large, there remain some eigenvalues which are far from this circle, since the law of large 
numbers involved in the proof of Proposition \ref{empirical} does not apply for the small cycles of the permutation
$\sigma_N$ associated with $M_N$. In order to study the influence of the small cycles, let us suppose
that $(M_N)_{N \geq 1}$ follows the measure $\mathbb{P}(\infty, \theta, \mathcal{L})$ (which is possible since
the image of $\mathbb{P}(\infty, \theta, \mathcal{L})$ by the $N$-th coordinate is $\mathbb{P}(N, \theta, 
\mathcal{L})$). Then, one can write the measure $\mu(M_N)$ in the following way:
$$\mu(M_N) = \sum_{k=1}^{\infty} \sum_{m \geq 1, l_{N,m} = k}  \sum_{\omega^k = Z_{N,m}} \delta_{\omega},$$
with the same notation as in equation \eqref{muMN}. This equality implies the following equality in
distribution:
\begin{equation}
\mu(M_N) = \sum_{k=1}^{\infty} \sum_{p = 1}^{a_{N,k}} \sum_{\omega^k = T_{k,p}} \delta_{\omega}, \label{muMN2}
\end{equation}
where for all $k \geq 1$, $a_{N,k}$ is the number of $k$-cycles of the permutation $\sigma_N$ (which follows
the Ewens($\theta$) measure on $\Sigma_N$), where for $k, p \geq 1$, the law of $T_{k,p}$ is
the multiplicative convolution of $k$ copies of the distribution $\mathcal{L}$, and where
$(a_{N,k})_{k \geq 1}$ and the variables $T_{k,p}$, $k, p \geq 1$ are independent. 
Now, the finite dimensional marginals of $(a_{N,k})_{k \geq 1}$ converge, in distribution, to the corresponding
marginales of $(a_k)_{k \geq 1}$, where the variables $a_k$ are independent Poisson random variables, 
with $\mathbb{E}[a_k] = \theta/k$ (see for instance \cite{ABT}). One can then expect that
in a sense which needs to be made precise, the law of $\mu(M_N)$ converges to the distribution
of
\begin{equation}
\mu_{\infty} :=  \sum_{k=1}^{\infty} \sum_{p=1}^{a_k} \sum_{\omega^k = T_{k,p}} \delta_{\omega}, \label{muinfty}
\end{equation}
where all the variables $a_k$ and $T_{k,p}$ in sight are  independent. 
Of course, one needs to be carefull, because the measure $\mu_{\infty}$ has an infinite total mass, which, 
under the assumptions of Proposition \ref{empirical}, is expected to concentrate around
 the circle of radius $\exp \left( \mathbb{E}[\log (|Z|)] \right)$. One also remarks that the
convergence expected here is very different from the convergence proved in Proposition \ref{empirical}; in 
particular, it involves the measure $\mu(M_N)$ and not the probability $\mu(M_N)/N$. 
In order to state a rigorously our result, let us give the following definition:
\begin{definition}
Let $X$ be a real, integrable, random variable. For $q > 0$, we say that $X$ is in $\mathcal{U}^q$ if and only if
for a sequence $(X_k)_{k \geq 1}$ of i.i.d. random variables with the same distribution as $X$ and for all
$\epsilon > 0$, there exists $C > 0$ such that for all $n \geq 1$:
$$\mathbb{P} \left[ \left| \left(\frac{1}{n} \sum_{k=1}^{n} X_k \right)  - \mathbb{E} [X] \right| \geq \epsilon
\right] \leq \frac{C}{n^q}. $$
\end{definition}
\begin{remark}
If $q \geq 1$ is an integer, by expanding $$\mathbb{E} \left[ \left( \frac{1}{n} \sum_{k=1}^{n} (X_k -
 \mathbb{E}[X])  \right)^{2q} \right] 
$$ and by using Markov's inequality, one easily proves that a random variable in $L^{2q}$ is also 
in $\mathcal{U}^q$.  
\end{remark}
\noindent
We can now state a result about the convergence of $\mu(M_N)$.
\begin{proposition} \label{feller}
Let $(M_N)_{N \geq 1}$ be a sequence of matrices such that $M_N \in \mathcal{G}(N)$ follows the distribution
 $\mathbb{P}(N, \theta, \mathcal{L})$
for some $\theta > 0$ and some probability $\mathcal{L}$ on $\mathbb{C}^*$. We suppose that if $Z$ is a random 
variable following the distribution $\mathcal{L}$, then $\log(|Z|)$ is in $\mathcal{U}^q$ for some $q > 0$. Under 
these 
assumptions, for all bounded, continuous functions $f$ from $\mathbb{C}$ to $\mathbb{R}$, such 
that $f= 0$ on a neighborhood of the circle $|z| = R$, where $R:= \exp \left( \mathbb{E}[\log (|Z|)] \right)$, 
 $f$ is a.s. integrable with respect to $\mu_{\infty}$,
where the random measure $\mu_{\infty}$ is given by \eqref{muinfty}. Moreover, the following
convergence in distribution holds for such test functions $f$:
$$\int_{\mathbb{C}} f \, d \mu(M_N) \underset{N \rightarrow \infty}{\longrightarrow} \int_{\mathbb{C}}
f \, d \mu_{\infty}.$$
\end{proposition}
\begin{proof}
Let $(\xi_r)_{r \geq 1}$ be a sequence of independent Bernouilli random variables, such that the 
parameter of $\xi_r$ is equal to $\theta/ (\theta + r-1)$ (in particular, $\xi_1 = 1$ almost surely).
We suppose that $(\xi_r)_{r \geq 1}$ is independent of $(T_{k,p})_{k,p \geq 1}$, and for all 
$N, k \geq 1$, we define $b_{N,k}$ as the number of pairs of consecutive ones in the sequence
$(\xi_1,\xi_2,...,\xi_N,1)$ which are separated by a distance of $k$, and $b_k$ as the analog
for the infinite sequence $(\xi_r)_{r \geq 1}$. In other words:
$$b_{N,k} = \mathds{1}_{N+1 -k \geq 1, \xi_{N+1-k} = 1, \, \xi_{N+2-k} = ... = \xi_{N} =0 }
 + \sum_{1 \leq j \leq N - k} 
\mathds{1}_{\xi_j = \xi_{j+k} = 1, \, \xi_{j+1} = ... = \xi_{j+k-1} = 0},$$
and
$$b_k = \sum_{j \geq 1} 
\mathds{1}_{\xi_j = \xi_{j+k} = 1, \, \xi_{j+1} = ... = \xi_{j+k-1} = 0}.$$
By the classical properties of the Feller coupling (see for instance \cite{ABT}), 
$(b_k)_{k \geq 1}$ has the same distribution as $(a_k)_{k \geq 1}$ and 
for all $N \geq 1$, $(b_{N,k})_{k \geq 1}$ has the same distribution as $(a_{N,k})_{k \geq 1}$.
Therefore, in Proposition \ref{feller}, one can replace $\mu(M_N)$ by $\nu_N$ and 
$\mu_{\infty}$ by $\nu_{\infty}$, where:
$$\nu_N = \sum_{k=1}^{\infty} \sum_{p = 1}^{b_{N,k}} \sum_{\omega^k = T_{k,p}} \delta_{\omega},$$
and 
$$\nu_{\infty} :=  \sum_{k=1}^{\infty} \sum_{p=1}^{b_k} \sum_{\omega^k = T_{k,p}} \delta_{\omega}.$$
Let $f$ be a continuous and bounded function from $\mathbb{C}$ to $\mathbb{R}$, equal to zero 
in a neighborhood of the circle $|z|= R$. There exists $A > 0$, $0 < R_1 < R < R_2$, depending
only on $f$, such that $|f(z)| \leq A \mathds{1}_{|z| \notin (R_1,R_2)}$, and then,
$$\int_{\mathbb{C}} |f| d \nu_{\infty} \leq 
A \sum_{k=1}^{\infty} \sum_{p=1}^{b_k} k \mathds{1}_{|T_{k,p}|^{1/k} \notin (R_1,R_2)}.$$
The function $f$ is a.s. integrable with repect to $\nu_{\infty}$ (and then, $\mu_{\infty}$). Indeed,
$$\int_{\mathbb{C}} |f| d \nu_{\infty} < \infty$$
if and only if
$$\sum_{k=1}^{\infty} \sum_{p=1}^{b_k} \mathds{1}_{|T_{k,p}|^{1/k} \notin (R_1,R_2)} <\infty,$$
and the expectation of this quantity is 
$$\theta \sum_{k=1}^{\infty} \frac{\mathbb{P} [|T_{k,1}|^{1/k} \notin (R_1,R_2)]}{k},$$
which is finite since $\log(|Z|)$ is in $\mathcal{U}^q$ for some $q > 0$. 
Let us now introduce the random measure:
$$\nu'_N = \sum_{k=1}^{\infty} \sum_{p= 1}^{c_{N,k}} \sum_{\omega^k = T_{k,p}} \delta_{\omega},$$
where 
\begin{equation}
c_{N,k} = \sum_{1 \leq j \leq N - k} 
\mathds{1}_{\xi_j = \xi_{j+k} = 1, \, \xi_{j+1} = ... = \xi_{j+k-1} = 0}. \label{cN}
\end{equation}
One has $b_{N,k} = c_{N,k}$, except for $k$ equal to the smallest integer $k_0$ such that $\xi_{N+1 - k_0} = 1$, in which case
 $b_{N,k_0} = c_{N,k_0} + 1$. Therefore, 
$$\nu_N = \nu'_N + \nu''_N,$$
where 
$$\nu''_N = \sum_{\omega^{k_0} = T_{k_0, c_{N,k_0} + 1}} \delta_{\omega}.$$
Since $f$ is a.s. integrable with respect to $\nu_{\infty}$, 
and for all $k \geq 1$, $b_k$ is the increasing limit of $c_{N,k}$, one obtains, a.s.:
\begin{equation}
\int_{\mathbb{C}} f \, d \nu'_N \underset{N \rightarrow \infty}{\longrightarrow}
 \int_{\mathbb{C}} f \, d \nu_{\infty}. \label{nu1}
\end{equation}
Moreover, there exist $0 < R_1 < R < R_2$, depending only on $f$, such that for all $k' \geq 1$:
\begin{align*}
\mathbb{P} \left[ \int_{\mathbb{C}} f \, d \nu''_N \neq 0 \right] & \leq \mathbb{P} 
\left[ |T_{k_0, c_{N, k_0} +1 }|^{1/k_0} \notin (R_1,R_2) \right]
\\ &  =  \sum_{k \geq 1} \mathbb{P} [k_0 = k] \mathbb{P} [|T_{k,1}|^{1/k} \notin (R_1,R_2) ] 
\\ & \leq \mathbb{P} [k_0 \leq k'] + \sup_{k'' > k'}  \mathbb{P} \left[|T_{k'', 1}|^{1/k''}
 \notin (R_1,R_2) \right]
\end{align*}
Now, $\mathbb{P} [k_0 \leq k']$ tends to zero when $N$ goes to infinity (the order of magnitude is $1/N$), 
and by taking $N \rightarrow \infty$, and then $k' \rightarrow \infty$, one deduces:
\begin{equation}
\mathbb{P} \left[ \int_{\mathbb{C}} f \, d \nu''_N \neq 0 \right]  \underset{N \rightarrow \infty}
{\longrightarrow} 0. \label{nu2}
\end{equation}
Therefore, by taking \eqref{nu1} and \eqref{nu2} together:
$$\int_{\mathbb{C}} f \, d \nu'_N \underset{N \rightarrow \infty}{\longrightarrow}
 \int_{\mathbb{C}} f \, d \nu_{\infty}$$
in probability, and a fortiori, in distribution. 
\end{proof}
\noindent
\begin{remark}
In the proof of Proposition \ref{feller}, we have used the Feller coupling in order to replace the convergence in law
by a convergence in probability. However, this coupling does not correspond to the coupling used 
by considering the measure $\mathbb{P} (\infty, \theta, \mathcal{L})$. Moreover, if 
$(M_N)_{N \geq 1}$ follows this measure, the number of cycles of a given length in $\sigma_N$ does not converge
when $N$ goes infinity, and the support of these cycles is changing infinitely often. Hence, one cannot expect 
an almost sure convergence (or even a convergence in probability) in Proposition \ref{feller}.
\end{remark}
\noindent
Propositions \ref{empirical} and \ref{feller} apply for a large family of distributions $\mathcal{L}$, however,
 some integrability conditions need to be satisfied. One can ask what happens if these conditions 
do not hold. We are not able to prove a result for all the possible distributions $\mathcal{L}$,
but it is possible to study some important particular cases, if the probability distributions 
involved in our problem can be explicitly computed. Here, the most important distributions we need to 
deal with are the probability measures $\mathcal{L}_k$, $k \geq 1$, defined in the following way:
$\mathcal{L}_k$ is the unique measure on $\mathbb{C}^*$, invariant by multiplication
 by a $k$-th root of unity, and such that its image by the $k$-th power is the multiplicative convolution of $k$ 
copies
 of $\mathcal{L}$. Intuitively, $\mathcal{L}_k$ is the law of a random $k$-th root of the product of $k$ independent random 
variables with law $\mathcal{L}$, chosen uniformly among the $k$ possible roots.  
The measures $\mathcal{L}_k$ are
not easy to compute in general. One case where the computation is simplified  is the case where $\mathcal{L}$ has the radial symmetry. Indeed let $(Y_p)_{p \geq 1}$ be a sequence 
of i.i.d. random variables with the same law as $\log(|Z|)$, where $Z$ follows the distribution $\mathcal{L}$. 
The distribution $\mathcal{L}_k$ is the law of $e^{i \Theta + \frac{1}{k} \sum_{p=1}^k Y_p}$
where $\Theta$ is independent of $(Y_p)_{p \geq 1}$ and uniform on $[0,2 \pi)$. In particular, 
if $\mathcal{L}$ is the law of $e^{i \Theta + \rho S_{\alpha}}$, where $\rho$ is a strictly positive parameter, 
$\Theta$ a uniform variable on $[0,2 \pi)$ and $S_{\alpha}$  ($\alpha \in (0,2]$) an independent 
standard symmetric stable random variable of index $\alpha$, 
then $\mathcal{L}_k$ is the law of $e^{i \Theta + \rho k^{(1-\alpha)/\alpha} S_{\alpha}}$.
Using this explicit description of $\mathcal{L}_k$, we can make a detailled study of the "stable case". 
For $\alpha > 1$ (and in particular for a log-normal modulus, corresponding to $\alpha=2$), Propositions
\ref{empirical} and \ref{feller} directly apply. Therefore, let us suppose $\alpha \leq 1$. For $\alpha =1$, one has the following:
\begin{proposition} \label{Cauchy}
 Let $\rho > 0$, and let $\mathcal{L}$ be the law of $e^{i \Theta + \rho S_{1}}$, where
$\Theta$ is a uniform random variable on $[0,2 \pi)$ and $S_{1}$ an independent
standard symmetric Cauchy random variable. 
For $\theta >0$, let $(M_N)_{N \geq 1}$ be a sequence of random matrices such that $M_N$ follows
the distribution $\mathbb{P} (N,\theta, \mathcal{L})$. Then the distribution of the random probability measure
 $\mu(M_N)/N$ converges to the law of the random measure
$$\bar{\mu}_{\infty} := \sum_{m \geq 1} x_m \mu^{U}_{e^{\rho S_{1}^{(m)}}},$$
where $(x_m)_{m \geq 1}$ is a Poisson-Dirichlet process with parameter $\theta$, $(S_1^{(m)})_{m \geq 1}$
is an independent sequence of i.i.d. standard symmetric Cauchy variables, and for $R > 0$, $\mu^{U}_{R}$
is the uniform measure on the circle of center zero and radius $R$. This convergence has to be understood
as follows: for all continuous and bounded functions $f$ from $\mathbb{C}$ to $\mathbb{R}$,
$$\frac{1}{N} \, \int_{\mathbb{C}} f \, d \mu(M_N) \underset{N \rightarrow \infty}{\longrightarrow}
\int_{\mathbb{C}} f \, d \bar{\mu}_{\infty}$$
in distribution. 
\end{proposition}
\begin{proof}
One can suppose that $(M_N)_{N \geq 1}$ follows the distribution $\mathbb{P} (\infty, \theta, \mathcal{L})$.
With the same notation as in the proof of Proposition \ref{empirical}, one sees that $\mu(M_N)/N$ has 
the same distribution as the random measure:
$$\bar{\mu}_N := \sum_{m = 1}^{\infty} \frac{l_{N,m}}{N} \, \left( \frac{\mathds{1}_{l_{N,m}>0} }{l_{N,m}}
\, \sum_{\omega \in \mathbb{U}_{l_{N,m}}} \delta_{\omega L_m} \right),$$
where $(L_m)_{m \geq 1}$ is a sequence of i.i.d. random variables with law $\mathcal{L}$ (recall that 
$\mathcal{L}_k$ is equal to $\mathcal{L}$ for all $k \geq 1$). 
If $f$ is a continuous, bounded function from $\mathbb{C}$ to $\mathbb{R}$, one has:
$$\int_{\mathbb{C}} f \, d \bar{\mu}_{N} = \sum_{m = 1}^{\infty}
  \frac{l_{N,m}}{N} \, \left( \frac{\mathds{1}_{l_{N,m}>0} }{l_{N,m}}
\, \sum_{\omega \in \mathbb{U}_{l_{N,m}}} f(\omega L_m) \right).$$
Now, there exists a GEM process $(y_m)_{m \geq 1}$ of parameter $\theta$ and a sequence of random variables
$(s_m)_{m \geq 1}$, such that almost surely,
$$ \sum_{m =1}^{\infty} s_m < \infty$$
$$\frac{l_{N,m}}{N} \underset{N \rightarrow \infty}{\longrightarrow} y_m,$$ 
for all $m \geq 1$ and 
$$\frac{l_{N,m}}{N} \leq s_m$$
for all $N, m \geq 1$. 
By the convergence of Riemann sums corresponding to the integral of continuous functions on a compact set, one
deduces that almost surely, for all $m \geq 1$:
$$\frac{l_{N,m}}{N} \, \left( \frac{\mathds{1}_{l_{N,m}>0} }{l_{N,m}}
\, \sum_{\omega \in \mathbb{U}_{l_{N,m}}} f(\omega L_m) \right) 
\underset{N \rightarrow \infty}{\longrightarrow} \frac{y_m}{2\pi}  \, \int_{0}^{2 \pi} f(L_m e^{i \lambda})
d \lambda,$$
where the left-hand side is smaller than or equal to $s_m ||f||_{\infty}$, independently of $N$. 
By dominated convergence:
$$\int_{\mathbb{C}} f \, d \bar{\mu}_{N} \underset{N \rightarrow \infty}{\longrightarrow} 
\frac{1}{2\pi} \, \sum_{m \geq 1} y_m  \, \int_{0}^{2 \pi} f(L_m e^{i \lambda}) d \lambda,$$
almost surely. This implies Proposition \ref{Cauchy}.
\end{proof}
\begin{remark}
Almost surely, the random measure $\bar{\mu}_{\infty}$ is strictly positive for any nonempty set of $\mathbb{C}$. 
Therefore, for all continuous functions $f$ from $\mathbb{C}$ to $\mathbb{R}_+$, non-identically zero, and 
for all $A \in \mathbb{R}_+$:
$$\mathbb{P} \left[\int_{\mathbb{C}} f \, d \mu(M_N) \leq A \right] \underset{N \rightarrow \infty}{\longrightarrow}
0.$$
Hence, one cannot expect an analog of Proposition \ref{feller} in the case studied here.
\end{remark}
\noindent
For the case $\alpha < 1$, one intuitively expects that most of the eigenvalues become very large or very small.
The precise statement is the following:
\begin{proposition} \label{alpha<1}
Let $\rho > 0$, and let $\mathcal{L}$ be the law of $e^{i \Theta + \rho S_{\alpha}}$, where
$\Theta$ is a uniform random variable on $[0,2 \pi)$ and $S_{\alpha}$ an independent
standard symmetric stable random variable, with index $\alpha <1$. 
For $\theta >0$, let $(M_N)_{N \geq 1}$ be a sequence of random matrices such that $M_N$ follows
the distribution $\mathbb{P} (N,\theta, \mathcal{L})$. Then, the distibution of the random probability
 $\mu(M_N)/N$ converges to the law of the random measure $G_{\theta} \delta_0$, where 
 $G_{\theta}$ is a beta random variable with parameters $(\theta/2,\theta/2)$, in the following 
sense: for all continuous functions $f$ from $\mathbb{C}$ to $\mathbb{R}$, with compact support,
$$\frac{1}{N} \, \int_{\mathbb{C}} f \, d \mu(M_N) \underset{N \rightarrow \infty}{\longrightarrow}
 G_{\theta} f(0)$$
in distribution.
\end{proposition}
\begin{remark}
The total mass of the limit measure is a.s. strictly smaller than one. Intuitively, this is due 
to the fact that a large part of the total mass of $\mu(M_N)$ is going far away from zero, when $N$ is large. 
This mass is missing in the limit measure in Proposition \ref{alpha<1}, because we consider
 functions $f$ with compact support.
\end{remark}
\begin{proof}
We suppose that $(M_N)_{N \geq 1}$ follows the distribution $\mathbb{P} (\infty, \theta, \mathcal{L})$. 
Let $f$ be a continuous function from $\mathbb{C}$ to $\mathbb{R}$, with compact support, 
and let us choose $R > 1$ such that $f(z) = 0$ for all $z$ such that $|z| > R$. Let us define, for  
all $z \in \mathbb{C}$, $$g(z) := \mathds{1}_{|z| \leq 1} f(0).$$
and for $r \in (0,1)$:
$$\beta(r) := \sup \{ |f(z) - f(0)|, |z| \leq r \},$$
which tends to zero with $r$. One checks that
$$\frac{1}{N} \, \int_{\mathbb{C}} |f-g| \, d\mu(M_N) \leq \beta(r) + 2 ||f||_{\infty} \frac{\mu(M_N) (\{z \in \mathbb{C}, 
r \leq |z| \leq R \})}{N},$$
which implies:
\begin{align*}
\mathbb{P} \left[ \frac{1}{N} \, \int_{\mathbb{C}} |f-g| \, d\mu(M_N) >  \beta(r) + 
(2 ||f||_{\infty} + 1) r \right]
& \leq \mathbb{P} \left[ \frac{\mu(M_N) (\{z \in \mathbb{C}, 
r \leq |z| \leq R \})}{N} \geq r \right] \\ & 
\leq \frac{1}{r} \mathbb{E} \left[ \frac{\mu(M_N) (\{z \in \mathbb{C}, 
r \leq |z| \leq R \})}{N} \right]. 
\end{align*}
\noindent
Now,
$$ \mathbb{E} \left[\frac{\mu(M_N) (\{z \in \mathbb{C}, 
r \leq |z| \leq R \})}{N} \right] = \sum_{m =1}^{\infty} \mathbb{E} \left[ \frac{l_{N,m}}{N}
\, \mathds{1}_{l_{N,m} > 0, |Z_{N,m}|^{(1/l_{N,m})} \in (r,R)} \right]. $$
From the independence of $l_{N,m}$ and $(z_j)_{j \geq 1}$, and from the basic properties 
of stable random variables, one has
$$ \mathbb{P} \left[ l_{N,m} > 0, |Z_{N,m}|^{(1/l_{N,m})} \in (r,R) | l_{N,m} \right]
= \Psi(l_{N,m}),$$
where the function $\Psi$ (which can depend on $\alpha$, $\rho$, $r$ and $R$, but not on $N$) is 
bounded by one and tends 
to zero at infinity.  Hence:
$$\mathbb{E} \left[\frac{\mu(M_N) (\{z \in \mathbb{C}, 
r \leq |z| \leq R \})}{N} \right] \leq \mathbb{E} \left[ \sum_{m =1}^{\infty} 
\frac{l_{N,m} \Psi(l_{N,m})}{N} \right].$$
Now, for all $m \geq 1$, a.s.:
$$\frac{l_{N,m} \Psi(l_{N,m})}{N}  \underset{N \rightarrow \infty}{\longrightarrow} 0$$
and 
$$\frac{l_{N,m} \Psi(l_{N,m})}{N} \leq s_m.$$
By dominated convergence,
$$ \mathbb{E} \left[\frac{\mu(M_N) (\{z \in \mathbb{C}, 
r \leq |z| \leq R \})}{N} \right] \underset{N \rightarrow \infty}{\longrightarrow} 0$$
which implies:
$$\mathbb{P} \left[ \frac{1}{N} \, \int_{\mathbb{C}} |f-g| \, d\mu(M_N) >  \beta(r) + (2||f||_{\infty} +1) r \right]
\underset{N \rightarrow \infty}{\longrightarrow} 0.$$
By letting $r$ go to zero, one deduces that
$$\frac{1}{N} \, \int_{\mathbb{C}} |f-g| \, d\mu(M_N) \underset{N \rightarrow \infty}{\longrightarrow} 0$$
in probability. 
Therefore, it is sufficient to prove the conclusion of Proposition \ref{alpha<1}, with $f$ replaced by $g$. 
Moreover, one can suppose $f(0) = g(0) = 1$. In this case, one has:
$$\frac{1}{N} \, \int_{\mathbb{C}} g  \, d\mu(M_N) 
= \sum_{m \geq 1} \frac{l_{N,m}}{N} \, \mathds{1}_{l_{N,m} > 0, |Z_{N,m}| \leq 1}.$$
Now, by symmetry of the stable variables considered here and by the independence of $(l_{N,m})_{m \geq 1}$ and
$(z_j)_{j \geq 1}$, 
$$\frac{1}{N} \, \int_{\mathbb{C}} g  \, d\mu(M_N) = 
\sum_{m \geq 1} \frac{l_{N,m}}{N} \, \epsilon_m$$
in distribution, where $(\epsilon_m)_{m \geq 1}$ is a sequence of i.i.d Bernoulli variables of parameter
1/2, independent of $(l_{N,m})_{N,m \geq 1}$. 
Now, by dominated convergence (recall that the sum of the variables $s_m$ is a.s. finite),
$$\sum_{m \geq 1} \frac{l_{N,m}}{N} \, \epsilon_m \underset{N \rightarrow \infty}{\longrightarrow}
\sum_{m \geq 1}  y_m \epsilon_m,$$
which implies Proposition \ref{alpha<1} if we check that 
$$X := \sum_{m \geq 1} y_m \epsilon_m$$
is equal to $G_{\theta}$ in distribution (recall 
that $(y_m)_{m \geq 1}$ is a GEM($\theta$) process, independent of 
the sequence $(\epsilon_m)_{m \geq 1}$). This fact can be proved as follows:
by self-similarity of the GEM($\theta$) process, 
$$X = VX + \epsilon (1-V),$$
in distribution, where $V$, $X$, $\epsilon$ are independent, $V$ is a beta variable of parameters $\theta$ and
$1$, and $\epsilon$ is a Bernoulli variable of parameter $1/2$. With this identity, one can compute the 
moments of $X$ by induction, and finally, one can identify its law. 
\end{proof}
\noindent
Again for the stable case with parameter $\alpha < 1$, one has an analog of Proposition \ref{feller}, 
for the eiganvalues which are in a compact set not containing zero. The precise statement is:
\begin{proposition}
Let $\rho > 0$, and let $\mathcal{L}$ be the law of $e^{i \Theta + \rho S_{\alpha}}$, where
$\Theta$ is a uniform random variable on $[0,2 \pi)$ and $S_{\alpha}$ an independent
standard symmetric stable random variable, with index $\alpha <1$. 
For $\theta >0$, let $(M_N)_{N \geq 1}$ be a sequence of random matrices such that $M_N$ follows
the distribution $\mathbb{P} (N,\theta, \mathcal{L})$. 
Then, for all continuous functions $f$ from $\mathbb{C}$ to $\mathbb{R}$, with compact support, and such 
that $f= 0$ on a neighborhood of zero, $f$ is a.s. integrable with respect to $\mu_{\infty}$,
where the random measure $\mu_{\infty}$ is given by \eqref{muinfty}. Moreover, the following
convergence in distribution holds with such test functions:
$$\int_{\mathbb{C}} f \, d \mu(M_N) \underset{N \rightarrow \infty}{\longrightarrow} \int_{\mathbb{C}}
f \, d \mu_{\infty}.$$
\end{proposition}
 \begin{proof}
The proof is similar to the proof of Proposition \ref{feller}. One only has  to change the 
estimate:
$$\mathbb{P} [|T_{k,1}|^{1/k} \notin (R_1,R_2)] \leq C k^{-q}$$
(for some $C, q > 0$), by the estimate:
$$\mathbb{P} [|T_{k,1}|^{1/k} \in (r,R)] \leq C k^{-q},$$
available for all $r$, $R$ such that $R > r > 0$, by the classical properties of symmetric stable 
random variables. 
\end{proof}
\subsection{The average eigenvalues distributions}
Another interesting problem is the study of the expectation $\tilde{\mu}_N$ of the random measure $\mu(M_N)$, where 
 $M_N$ follows the distribution $\mathbb{P} (N, \theta, \mathcal{L})$. The measure $\tilde{\mu}_N$ of 
a Borel set in $\mathbb{C}$ represents the expected number of eigenvalues of $M_N$ (with multiplicity)
contained in this set, and it can be explicitly computed:
\begin{proposition} \label{mutilde}
Let $N \geq 1$, $\theta > 0$, $\mathcal{L}$ a probability measure on $\mathbb{C}^*$. If $M_N$ follows 
the distribution $\mathbb{P} (N, \theta, \mathcal{L})$, then the expectation $\tilde{\mu}_N$ of 
$\mu(M_N)$ can be represented as follows:
$$\tilde{\mu}_N = \theta \sum_{k=1}^N \frac{N(N-1)...(N-k+1)}{(N-1+ \theta)(N-2+\theta)...(N-k+
\theta)} \, \mathcal{L}_k.$$
\end{proposition}
\begin{proof}
Let $f$ be a Borel function from $\mathbb{C}$ to $\mathbb{R}_+$. By taking the same notation as
in \eqref{muMN2}, one has:
$$\int_{\mathbb{C}} f \, d \tilde{\mu}_N = \mathbb{E} \left[ 
\sum_{k=1}^{\infty} \sum_{p = 1}^{a_{N,k}} \sum_{\omega^k = T_{k,p}} f(\omega) \right]$$
Now, for all $k, p \geq 1$:
$$\mathbb{E} \left[ \sum_{\omega^k = T_{k,p}} f(\omega) \right] = k \int_{\mathbb{C}} f \, d\mathcal{L}_k$$
Moreover, $(T_{k,p})_{k, p \geq 1}$ is independent of $(a_{N,k})_{k \geq 1}$. One deduces that 
\begin{align*}
 & \; \mathbb{E} \left[ 
\sum_{k=1}^{\infty} \sum_{p = 1}^{a_{N,k}} \sum_{\omega^k = T_{k,p}} f(\omega) | (a_{N,k})_{k \geq 1} 
 \right] \\ & = \sum_{k=1}^{\infty} k \, a_{N,k} \int_{\mathbb{C}} f \, d\mathcal{L}_k.
\end{align*}
Therefore
$$\tilde{\mu}_N = \sum_{k=1}^{\infty} k \, \mathbb{E} [a_{N,k}] \, \mathcal{L}_k.$$
By doing explicit computations of $\mathbb{E} [a_{N,k}]$ (see e.g. \cite{ABT}), one deduces Proposition \ref{mutilde}.
\end{proof}
\begin{remark}
When the distribution of the $z_j$'s is the Dirac mass at $1$, then $\mathcal{L}_k$ is the uniform measure on the $k$-th roots of unity.
\end{remark}
One has a similar result for the limit random measure $\mu_{\infty}$ of $\mu(M_N)$:
\begin{proposition} \label{mutildeinfty}
Let $\theta > 0$, $\mathcal{L}$ a probability measure on $\mathbb{C}^*$, $(M_N)_{N \geq 1}$
a sequence of random matrices such that $M_N$ follows the distribution $\mathbb{P} (N, \theta, 
\mathcal{L})$. Then the expectation of the random measure $\mu_{\infty}$, given by 
 \eqref{muinfty} (it
can be considered, in some sense, as the limit of $\mu(M_N)$ for $N \rightarrow \infty$), is the measure:
 $$\tilde{\mu}_{\infty} = \theta \sum_{k=1}^{\infty} \mathcal{L}_k.$$
\end{proposition}
\begin{proof}
The proof of Proposition \ref{mutildeinfty} is exactly similar to the proof of Proposition \ref{mutilde}.
\end{proof}
Now, we have a sequence of finite measures $\tilde{\mu}_{N}$, defined as the expectation
of $\mu(M_N)$ and explicitly described, and an infinite measure $\tilde{\mu}_{\infty}$, defined as the
 expectation of $\mu_{\infty}$. Moreover, we know that, for a large class of probability 
laws $\mathcal{L}$, the random probability measure $\mu(M_N)/N$ converges weakly in probability 
to the uniform measure on a circle, and in a sense which can be made precise, 
$\mu(M_N)$ tends to $\mu_{\infty}$. Hence, we can expect analog convergences for  the 
sequence of measures $(\tilde{\mu}_N)_{N \geq 1}$. One indeed has the following result:
\begin{proposition} \label{convmutilde}
Let $\theta > 0$, $\mathcal{L}$ a probability measure on $\mathbb{C}^*$, $(M_N)_{N \geq 1}$
a sequence of random matrices such that $M_N$ follows the distribution $\mathbb{P} (N, \theta, 
\mathcal{L})$. We suppose that if $Z$ follows the distribution $\mathcal{L}$,
then $\log (|Z|)$ is integrable, and we define $R := \exp \left( \mathbb{E}[\log (|Z|)] \right)$.
Then the probability measure $\tilde{\mu}_N/N$ (which represents the probability distribution
of a random eigenvalue of $M_N$, chosen uniformly among the $N$ possible eigenvalues), converges
weakly to the uniform distribution on the circle $\{z = R\}$. Moreover, if $\log (|Z|)$ is in 
$\mathcal{U}^2$, then for $0 < R_1 < R < R_2$, the restriction of $\tilde{\mu}_N$ to the set $\{|z| \notin (R_1,R_2) \}$
converges weakly to the corresponding restriction of $\tilde{\mu}_{\infty}$, which is a finite measure. 
\end{proposition}
\begin{proof}
Let $f$ be a continuous and bounded function from $\mathbb{C}$ to $\mathbb{R}$. By Proposition
\ref{empirical} (more precisely, by Corollary \ref{empirical2}):
$$\frac{1}{N} \int_{\mathbb{C}} f \, d\mu(M_N) \underset{N \rightarrow \infty}{\longrightarrow} 
\frac{1}{2 \pi} \int_0^{2 \pi} f(Re^{i \lambda}) d \lambda$$
in probability. Since $f$ is uniformly bounded, one deduces the first part of Proposition \ref{convmutilde}
by taking the expectation. Let us now prove the second part. We now suppose that 
$f$ is a bounded and continuous function from $\mathbb{C}$ to $\mathbb{R}$, equal to zero in 
a neighborhood of the circle $\{|z| = R\}$. By Proposition \ref{mutilde}, it is sufficient to prove that
$$\sum_{k=1}^{\infty} |t_{N,k, \theta} \mathds{1}_{N \geq k} - 1| \int_{\mathbb{C}} |f| \, d\mathcal{L}_k \underset{N \rightarrow \infty}
{\longrightarrow} 0,$$
where 
$$t_{N,k, \theta} = \frac{N(N-1)...(N-k+1)}{(N-1 + \theta)(N-2+ \theta)...(N-k+ \theta)}.$$
Each term of the sum converges to zero when $N$ goes to infinity. Hence by dominated convergence we are done
if we prove that:
\begin{equation}
\sum_{k=1}^{\infty} v_{k, \theta} \int_{\mathbb{C}} |f| \, d\mathcal{L}_k < \infty, \label{finite}
\end{equation}
where $$v_{k, \theta} = 1 + \sup\{ t_{N,k, \theta}, N \geq k \}.$$
Now, for $\theta \geq 1$, $v_{k, \theta} = 2$, and for $\theta < 1$:
$$v_{k, \theta} = 1 + \frac{k!(\theta-1)!}{(k-1+\theta)!},$$
which implies that $v_{k, \theta}$ is dominated by $k^{(1- \theta)_+}$ for fixed $\theta$.
On the other hand, the fact that $\log(|Z|)$ is in $\mathcal{U}^2$ implies that the integral of 
$|f|$ with respect to $\mathcal{L}_k$ decreases with $k$ at least as fast as $1/k^2$. 
This implies \eqref{finite}.
\end{proof}
\noindent
As above, it is interesting to see what happens if $\mathcal{L}$ is the distribution of
$e^{i \Theta + \rho S_{\alpha}}$, where $\rho > 0$, 
$\Theta$ is a uniform random variable on $[0,2 \pi)$ and $S_{\alpha}$ is an independent
standard symmetric stable random variable, with index $\alpha \in (0,2]$.
For the log-normal case $\alpha = 2$, Proposition \ref{convmutilde} applies directly. 
For the case $\alpha \in (1,2)$, one can apply the first part of Proposition \ref{convmutilde}, 
which gives the convergence of the average empirical measure $\tilde{\mu}_N/N$, but the second 
part cannot apply. Indeed, by using the classical tail estimates of stable random variables
 one checks that for all nonempty open sets $A \subset \mathbb{C}$:
$$\tilde{\mu}_{\infty} (A) = \theta \sum_{k=1}^{\infty} \mathcal{L}_k (A) = \infty.$$
For $\alpha = 1$, one has $\mathcal{L}_k = \mathcal{L}$ for all $k \geq 1$, and then:
$$\tilde{\mu}_N = N \mathcal{L}$$
for all $N \geq 1$. The most interesting case is $\alpha < 1$.
\begin{proposition} \label{tildealpha<1}
Let $\rho > 0$, and let $\mathcal{L}$ be the law of $e^{i \Theta + \rho S_{\alpha}}$, where
$\Theta$ is a uniform random variable on $[0,2 \pi)$ and $S_{\alpha}$ an independent
standard symmetric stable random variable, with index $\alpha <1$. 
For $\theta >0$, let $(M_N)_{N \geq 1}$ be a sequence of random matrices such that $M_N$ follows
the distribution $\mathbb{P} (N,\theta, \mathcal{L})$. 
Then the probability measure $\tilde{\mu}_N/N$ (which represents the probability distribution
of a random eigenvalue of $M_N$, chosen uniformly among the $N$ possible eigenvalues), converges
vaguely to half of the Dirac measure at zero. Moreover, for all $r, R$ such that $0 < r < R$, the 
 restriction of $\tilde{\mu}_{\infty}$ to the set $\{|z| \in (r,R) \}$ is infinite if $\alpha \geq 1/2$, 
finite if $\alpha < 1/2$, and for $\alpha < 1/(2 \vee (3-\theta))$ (in particular for $\alpha \leq 1/3$), it
 is the weak limit of the 
corresponding restriction of $\tilde{\mu}_N$ when $N$ goes to infinity. 
\end{proposition}
\begin{proof}
Let $f$ be a continuous function from $\mathbb{C}$ to $\mathbb{R}$, with compact support. 
By Proposition \ref{alpha<1}, one has the convergence in distribution:
$$\frac{1}{N} \, \int_{\mathbb{C}} f \, d \mu(M_N) \underset{N \rightarrow \infty}{\longrightarrow}
 G_{\theta} f(0).$$
Since $f$ is uniformly bounded, one can take the expectation and obtain:
$$\frac{1}{N} \, \int_{\mathbb{C}} f \, d \tilde{\mu}_N \underset{N \rightarrow \infty}{\longrightarrow}
 \frac{1}{2} \, f(0),$$
which gives the first part of Proposition \ref{tildealpha<1}. Now, one checks that for all $r, R$ such that 
$0<r<R$, $\mathcal{L}_k ( \{ z \in \mathbb{C}, |z| \in (r,R) \})$ decreases as $k^{1 - 1/ \alpha}$ when $k$ goes 
to infinity, which gives the condition for the finiteness of $\tilde{\mu}_{\infty}$ restricted 
to $\{z \in \mathbb{C}, |z| \in (r,R)\}$. The proof 
of the convergence of the restriction of $\tilde{\mu}_N$ to this set toward the restriction of
 $\tilde{\mu}_{\infty}$ is similar to the proof of the second part of Proposition \ref{convmutilde}. One 
only needs to check that:
$$\sum_{k =1}^{\infty} v_{k, \theta} \mathcal{L}_k ( \{ z \in \mathbb{C}, |z| \in (r,R) \}) < \infty,$$
which is true, since $v_{k, \theta}$ is dominated by $k^{(1-\theta)_+}$ and 
$ \mathcal{L}_k ( \{ z \in \mathbb{C}, |z| \in (r,R) \})$ is dominated by $k^{1 - 1/\alpha}$, 
for $1/\alpha > 2 \vee (3-\theta)$.
\end{proof}
\subsection{The $q$-correlation}
The measures $\tilde{\mu}_{N}$ and $\tilde{\mu}_{\infty}$ give the average repartition of 
the eigenvalues (and its limit for large dimension) of a matrix following the
distribution $\mathbb{P} (N, \theta, \mathcal{L})$. It is interesting to generalize this study
to the $q$-correlation of eigenvalues for all strictly positive integers $q$, i.e. the distribution
 of the possible sequences of $q$ eigenvalues. More precisely, if $M_N$ follows the distribution 
$\mathbb{P}(N, \theta, \mathcal{L})$, let us consider, for all $q \geq 1$, the random
measure on $\mathbb{C}^q$:
$$\mu^{[q]}(M_N) = \sum_{j_1 \neq j_2 \neq ... \neq j_q} \delta_{(\omega_{j_1},..., \omega_{j_q})},$$
where $(\omega_j)_{1 \leq j \leq N}$ is a sequence containing all the eigenvalues of $M_N$, with multiplicity. 
The "$q$-correlation measure" is defined as the average of $\mu^{[q]}(M_N)$: it is a finite measure
 $\tilde{\mu}_N^{[q]}$ on $\mathbb{C}^q$, with total 
mass $N!/(N-q)!$. One checks that $\mu^{[q]}(M_N)$ has (by taking the same notation as in \eqref{muMN2})
the same distribution as the random measure
$$\nu_N  := \sum_{(w_{k,p})_{k \geq 1, 1 \leq p \leq a_{N,k}}  \in W\left(q, (a_{N,k})_{k \geq 1}\right)}
\sum_{\left(E_{k,p} \in S_{k,w_{k,p}}(T_{k,p}) \right)_{k \geq 1, 1 \leq p \leq a_{N,k}}} \Delta \left(
(E_{k,p})_{k \geq 1, 1 \leq p \leq a_{N,k}} \right),$$
where $W(q, (a_{N,k})_{k \geq 1})$ is the set of families $(w_{k,p})_{k \geq 1, 1 \leq p \leq a_{N,k}}$
of nonnegative integers, with total sum $q$, $S_{k,w_{k,p}}(T_{k,p})$ is the family of subsets 
of $k$-th roots of $T_{k,p}$, with cardinality $w_{k,p}$, and $ \Delta \left(
(E_{k,p})_{k \geq 1, 1 \leq p \leq a_{N,k}} \right)$ is defined by:
$$\Delta \left(
(E_{k,p})_{k \geq 1, 1 \leq p \leq a_{N,k}} \right) = \sum_{\sigma \in \Sigma_q}
\delta_{(\omega_\sigma(j))_{1 \leq j \leq q}}$$
for a sequence $(\omega_j)_{1 \leq j \leq q}$ containing all the elements of the sets $E_{k,p}$, each
 element appearing a number of times equal to the number of sets $E_{k,p}$ containing it. 
Let us now take the conditional expectation with respect to $(a_{N,k})_{k \geq 1}$. One obtains the following
random measure:
$$\mathbb{E} [\nu_N \, | \, (a_{N,k})_{k \geq 1} ] = 
\sum_{(w_{k,p})_{k \geq 1, 1 \leq p \leq a_{N,k}}  \in W\left(q, (a_{N,k})_{k \geq 1}\right)}
\mathcal{L}^{[q]} \left( (w_{k,p})_{k \geq 1, 1 \leq p \leq a_{N,k}} \right),$$
for
$$\mathcal{L}^{[q]} \left( (w_{k,p})_{k \geq 1, 1 \leq p \leq a_{N,k}} \right)
 = \frac{1}{\prod_{k \geq 1, 1 \leq p \leq a_{N,k}} w_{k,p}!}  
\, \sum_{\sigma \in \Sigma_q} \sigma . \tilde{\mathcal{L}}^{[q]} \left( (w_{k,p})_{k \geq 1, 1 \leq p
 \leq a_{N,k}} \right),$$
where $$\sigma . \tilde{\mathcal{L}}^{[q]} \left( (w_{k,p})_{k \geq 1, 1 \leq p
 \leq a_{N,k}} \right)$$ is the image, by the permutation $\sigma$ of the coordinates, of 
$$\tilde{\mathcal{L}}^{[q]} \left( (w_{k,p})_{k \geq 1, 1 \leq p
 \leq a_{N,k}} \right) = \bigotimes_{k \geq 1, 1 \leq p \leq a_{N,k}} \mathcal{L}_k^{[w_{k,p}]},
$$
for 
$$\mathcal{L}_k^{[w_{k,p}] } = \sum_{z_1 \neq ... \neq z_{w_{k,p}} \in \mathbb{U}_k} 
(z_1,z_2,...,z_{w_{k,p}}). \mathcal{L}_k,$$
where $(z_1,z_2,...,z_{w_{k,p}}). \mathcal{L}_k$ is the image of $\mathcal{L}_k$ by the
 application $z \mapsto (zz_1,zz_2,...,zz_{w_{k,p}})$. By defining $\lambda_{k,r}$ as 
the number of indices $p$ such that $w_{k,p} = r$, one deduces:
$$ \mathbb{E} [\nu_N \, | \, (a_{N,k})_{k \geq 1} ] = 
\sum_{(\lambda_{k,r})_{k \geq 1, r \geq 0} \in L\left(q, (a_{N,k})_{k \geq 1}\right)} \frac{
\prod_{k \geq 1} a_{N,k}!}{\prod_{k \geq 1, r \geq 0} \left[\lambda_{k,r}! (r!)^{\lambda_{k,r}} \right]}
 \sum_{\sigma \in \Sigma_q}  \sigma. \left[\bigotimes_{k \geq 1,r \geq 0}
 \left( \mathcal{L}_k^{[r]} \right)^{\otimes \lambda_{k,r}} \right]
$$
where $L\left(q, (a_{N,k})_{k \geq 1}\right)$ is the set of families $(\lambda_{k,r})_{k \geq 1,r \geq 0}$ 
of nonnegative integers such that
$$\sum_{k \geq 1, r \geq 0} r \lambda_{k,r} = q$$
and for all $k \geq 1$:
$$\sum_{r \geq 0} \lambda_{k,r} = a_{N,k}.$$
Note that for all $k \geq 1$, the measure $\mathcal{L}_k^{[0]}$ is trivial (it is the 
unique probability measure on a space with one element), hence,  one can remove it in a tensor product.
One deduces:
\begin{align*}
\mathbb{E} [\nu_N \, | \, (a_{N,k})_{k \geq 1} ] & = 
\sum_{(\lambda_{k,r})_{k, r \geq 1} \in \tilde{L}(q)} \prod_{k \geq 1} \frac{
 a_{N,k}!}{ \left( a_{N,k} - \sum_{r \geq 1}
 \lambda_{k,r} \right)! \prod_{r \geq 1} \left[\lambda_{k,r}! (r!)^{\lambda_{k,r}} \right]} \;... \\
& ... \; \sum_{\sigma \in \Sigma_q}  \sigma. \left[\bigotimes_{k \geq 1,r \geq 1}
 \left( \mathcal{L}_k^{[r]} \right)^{\otimes \lambda_{k,r}} \right]
\end{align*}
\noindent
where $\tilde{L}(q)$ is the set of families $(\lambda_{k,r})_{k,r \geq 1}$ of nonnegative integers such
that
$$\sum_{k \geq 1, r \geq 1} r \lambda_{k,r} = q$$
and where the inverse of the factorial of a strictly negative integer is considered to be equal to zero. 
Then:
$$\tilde{\mu}_N^{[q]} = \sum_{(\lambda_{k,r})_{k, r \geq 1} \in \tilde{L}(q)}
u^{\left[\sum_{r \geq 1} \lambda_{k,r} \right]_{k \geq 1}}_{N, \theta} 
 \frac{1}
{\prod_{k, r \geq 1} \left[\lambda_{k,r}! (r!)^{\lambda_{k,r}} \right]} \;
 \sum_{\sigma \in \Sigma_q}  \sigma. \left[\bigotimes_{k \geq 1,r \geq 1}
 \left( \mathcal{L}_k^{[r]} \right)^{\otimes \lambda_{k,r}} \right],$$
where, for all sequences of nonnegative integers $(\lambda_k)_{k \geq 1}$
for which the set of $k$ such that $\lambda_k > 0$ is finite:
 $$u^{[\lambda_k]_{k \geq 1}}_{N, \theta} = \mathbb{E} \left[ \prod_{k \geq 1}
\frac{a_{N,k}!}{(a_{N,k} - \lambda_k)!} \right]$$ 
(recall that this quantity depends on $\theta$, as the law of $(a_{N,k})_{k \geq 1}$). 
 By elementary combinatorial arguments, one can prove that 
$$u^{[\lambda_k]_{k \geq 1}}_{N, \theta} = \frac{\prod_{j = 1}^{\sum_{k \geq 0} \lambda_k} 
(N+1-j)} { \prod_{k \geq 1} k^{\lambda_k} } \, . \mathcal{P},$$
where $\mathcal{P}$ is the probability that in a random permutation following the Ewens($\theta$) distribution,
the integers from $1 + \sum_{j < k} \lambda_j$ and $\sum_{j \leq k} \lambda_j$ lie
in different $k$-cycles, for all $k \geq 1$. By the Feller coupling, one deduces that
$$u^{[\lambda_k]_{k \geq 1}}_{N, \theta} = \mathds{1}_{\sum_{k \geq 1} k \lambda_k \leq N}
\left( \prod_{j=1}^{\sum_{k \geq 1} k \lambda_k}  \, \frac{N+1-j}{N+\theta-j}  \right) \,
 \frac{\theta^{\sum_{k \geq 1} \lambda_k} }{\prod_{k \geq 1} k^{\lambda_k} }.$$
Then by denoting $\tilde{L} (N,q)$ the set of families $(\lambda_{k,r})_{k,r \geq 1}$ of nonnegative
 integers such that  
$$\sum_{k \geq 1, r \geq 1} r \lambda_{k,r} = q$$
and 
$$\sum_{k \geq 1, r \geq 1} k \lambda_{k,r} \leq N,$$
one obtains the following result:
\begin{proposition} \label{q28}
Let $\theta > 0$ and let $\mathcal{L}$ be a probability distribution on $\mathbb{C}^*$. 
If one takes the notation above, the $q$-correlation measure
 $\tilde{\mu}_N^{[q]}$  associated with the eigenvalues of a random matrix following
the distribution $\mathbb{P} (N, \theta, \mathcal{L})$ is given by the formula:
\begin{align*}
\tilde{\mu}_N^{[q]} = \sum_{(\lambda_{k,r})_{k, r \geq 1} \in \tilde{L}(N,q)} & 
 \left( \prod_{j=1}^{\sum_{k,r \geq 1} k \lambda_{k,r} } \, \frac{N+1-j}{N+\theta-j}  \right)  \, \times \\ & ...
\prod_{k, r \geq 1} \left[\frac{1}{\lambda_{k,r}!} \left( \frac{\theta}{r!k} \right)^{\lambda_{k,r}} \right]  \;
 \sum_{\sigma \in \Sigma_q}  \sigma. \left[\bigotimes_{k \geq 1,r \geq 1}
 \left( \mathcal{L}_k^{[r]} \right)^{\otimes \lambda_{k,r}} \right]
\end{align*}
\end{proposition}
\begin{remark}
One can check that Proposition \ref{mutilde} is a particular case of Proposition \ref{q28} for $q=1$. Note that for $N, q \geq 2$, the $q$-correlation measure $\tilde{\mu}_N^{[q]}$ is not absolutely continuous with respect to 
Lebesgue measure. This means that the law of the point process of the eigenvalues cannot be described 
with correlation functions, as for determinantal processes.
\end{remark}
\noindent
In a similar way, we can define the $q$-correlation measure $\tilde{\mu}_{\infty}^{[q]}$, 
associated with the random point measure $\mu_{\infty}$ defined by \eqref{muinfty}. We obtain
the following:
\begin{proposition}
Let $\theta > 0$ and $\mathcal{L}$ be a probability distribution on $\mathbb{C}^*$. 
Then, the $q$-correlation measure
 $\tilde{\mu}_{\infty}^{[q]}$  associated with the point measure $\mu_{\infty}$ defined by \eqref{muinfty}
is given by:
$$
\tilde{\mu}_{\infty}^{[q]} = \sum_{(\lambda_{k,r})_{k, r \geq 1} \in \tilde{L}(q)} 
 \prod_{k, r \geq 1} \left[\frac{1}{\lambda_{k,r}!} \left( \frac{\theta}{r!k} \right)^{\lambda_{k,r}} \right]  \;
 \sum_{\sigma \in \Sigma_q}  \sigma. \left[\bigotimes_{k \geq 1,r \geq 1}
 \left( \mathcal{L}_k^{[r]} \right)^{\otimes \lambda_{k,r}} \right].$$
\end{proposition}
\noindent
Since we have an explicit expression for the correlations measures, we can expect some limit theorems 
when the dimension $N$ goes to infinity. In fact, we have the following proposition:
\begin{proposition} \label{convmutildeq}
Let $\theta > 0$, $q \geq 1$ integer, $\mathcal{L}$ a probability measure on $\mathbb{C}^*$, $(M_N)_{N \geq 1}$
a sequence of random matrices such that $M_N$ follows the distribution $\mathbb{P} (N, \theta, 
\mathcal{L})$. We suppose that if $Z$ follows the distribution $\mathcal{L}$,
then $\log (|Z|)$ is integrable, and we define $R := \exp \left( \mathbb{E}[\log (|Z|)] \right)$.
Then the probability measure $\frac{1}{N(N-1)...(N-q+1)} \, \tilde{\mu}_N^{[q]}$ (which represents 
the probability distribution
of a random sequence of $q$ eigenvalues of $M_N$, chosen uniformly among the $N(N-1)...(N-q+1)$ possible
 sequences), converges
weakly to the $q$-th power of the uniform distribution on the circle $\{z = R\}$. Moreover, if $\log (|Z|)$ is in 
$\mathcal{U}^{q+1}$, then for $0 < R_1 < R < R_2$, the restriction of $\tilde{\mu}^{[q]}_N$ to the set
$\{(z_1,...,z_q) \in \mathbb{C}^q, \forall r \leq q, |z_r| \notin (R_1,R_2)\}$, 
converges weakly to the corresponding restriction of $\tilde{\mu}^{[q]}_{\infty}$. 
\end{proposition}
\begin{proof}
Recall that $\tilde{\mu}_N^{[q]}$ is the average of the random measure:
$$\mu^{[q]}(M_N) = \sum_{j_1 \neq j_2 \neq ... \neq j_q} \delta_{(\omega_{j_1},..., \omega_{j_q})},$$
where $(\omega_j)_{1 \leq j \leq N}$ is a sequence containing all the eigenvalues of $M_N$, with multiplicity.
Let $f_1,...,f_q$ be bounded, continuous functions from $\mathbb{C}$ to $\mathbb{R}_+$. One has:
$$I := \int_{\mathbb{C}^q} f_1(z_1)...f_q(z_q) \, d \tilde{\mu}_N^{[q]} (z_1,...,z_q)
=  \mathbb{E} \left[\sum_{j_1 \neq j_2 \neq ... \neq j_q} f_1(\omega_{j_1})...f_q(\omega_{j_q}) \right],$$
which implies, 
\begin{equation}
 J - ||f_1||_{\infty}... ||f_q||_{\infty} [N^q - N(N-1)...(N-q+1)] \leq I \leq J, \label{IJ}
\end{equation}
for 
$$J := \mathbb{E} \left[ \prod_{r=1}^q \left( \sum_{j=1}^N  f_r(\omega_j) \right) \right]$$
or equivalently
$$J = \mathbb{E} \left[ \prod_{r=1}^q  \left( \int_{\mathbb{C}} f_r \, d \mu(M_N) \right) \right].$$
Now, by Proposition \ref{empirical} and Corollary \ref{empirical2}, 
$$\left( \frac{1}{N} \int_{\mathbb{C}} f_r  \, d \mu(M_N) \right)_{1 \leq r \leq q} 
\underset{N \rightarrow \infty}{\longrightarrow} \left( \frac{1}{2\pi} \, \int_{0}^{2 \pi} f_r(Re^{i \lambda}) d
\lambda \right)_{1 \leq r \leq q}$$
in probability. By applying the bounded, continuous function from $\mathbb{R}^q$ to $\mathbb{R}$:
$$ (x_1,...,x_q) \mapsto (|x_1|...|x_q|) \wedge ( ||f_1||_{\infty}... ||f_q||_{\infty}),$$
one deduces that
$$ \frac{J}{N^q} \, \underset{N \rightarrow \infty}{\longrightarrow} \,
\frac{1}{(2 \pi)^q} \, \int_{[0,2 \pi]^q} f_1(Re^{i \lambda_1})...f_q(Re^{i \lambda_q}) d\lambda_1...d\lambda_q.$$
By using the inequalities \eqref{IJ}, one
obtains that
$$ \frac{I}{N(N-1)...(N-q+1)} \, \underset{N \rightarrow \infty}{\longrightarrow} \,
\frac{1}{(2 \pi)^q} \, \int_{[0,2 \pi]^q} f_1(Re^{i \lambda_1})...f_q(Re^{i \lambda_q}) d\lambda_1...d\lambda_q,$$
which gives the first part of Proposition \ref{convmutildeq}.
In order to prove the second part, let us suppose that $0<R_1<R<R_2$ and that $f$ is a bounded, 
continuous function from $\mathbb{C}^q$ to $\mathbb{R}$, vanishing if one of the coordinates has a modulus
between $R_1$ and $R_2$. It is sufficient to prove:
\begin{align*}
& \sum_{(\lambda_{k,r})_{k, r \geq 1} \in \tilde{L}(q)} \left|\mathds{1}_{N \geq \sum_{k, r \geq 1} k \lambda_{k,r}}
 \; t_{N,\sum_{k, r \geq 1} k \lambda_{k,r}, \theta} - 1 \right| \, \times \\ &  \qquad...
 \prod_{k, r \geq 1} \left[\frac{1}{\lambda_{k,r}!} \left( \frac{\theta}{r!k} \right)^{\lambda_{k,r}} \right]  \;
 \sum_{\sigma \in \Sigma_q}  \int_{\mathbb{C}^q} f \, d \left( \sigma. \left[\bigotimes_{k \geq 1,r \geq 1}
 \left( \mathcal{L}_k^{[r]} \right)^{\otimes \lambda_{k,r}} \right] \right) \underset{N \rightarrow 
\infty}{\longrightarrow} 0,
\end{align*}
\noindent
with the notation of the proof of Propositions \ref{convmutilde} and \ref{q28}. By dominated convergence, we are done if we
check:
\begin{align*}
& \sum_{(\lambda_{k,r})_{k, r \geq 1} \in \tilde{L}(q)} v_{\sum_{k, r \geq 1} k \lambda_{k,r}, \theta}
\, \times \\ &  \qquad...
 \prod_{k, r \geq 1} \left[\frac{1}{\lambda_{k,r}!} \left( \frac{\theta}{r!k} \right)^{\lambda_{k,r}} \right]  \;
 \sum_{\sigma \in \Sigma_q}  \int_{\mathbb{C}^q} f \, d \left( \sigma. \left[\bigotimes_{k \geq 1,r \geq 1}
 \left( \mathcal{L}_k^{[r]} \right)^{\otimes \lambda_{k,r}} \right] \right) < \infty.
\end{align*}
\noindent
Now, for $(\lambda_{k,r})_{k, r \geq 1} \in \tilde{L}(q)$, 
  by the estimates obtained in the proof of Proposition \ref{convmutilde}:
\begin{align*}
v_{\sum_{k, r \geq 1} k \lambda_{k,r}, \theta} & \leq C(\theta)
 \left(1 + \sum_{k, r \geq 1} k \lambda_{k,r}\right)^{(1- 
\theta)_+} \\ &  \leq 2 q C(\theta) \left[ 1 \, \vee \, \sup \{k \geq 1, \exists r \geq 1, \lambda_{k,r} > 0 \}
\right]^{(1-\theta)_+} \\ & \leq 2 q C(\theta)  \prod_{k, r \geq 1} k^{\lambda_{k,r}(1-\theta)_+}, 
\end{align*}
\noindent
where $C(\theta)> 0 $ depends only on $\theta$. 
By assumption, if $Z$ is a random variable which follows the distribution $\mathcal{L}$, then $\log (Z) \in \mathcal{U}^{q+1}$, which implies that
$$\mathcal{L}_k \{z \in \mathbb{C}, 
|z| \notin (R_1,R_2)\}$$
decreases at least as fast as $1/k^{q+1}$ when $k$ goes to infinity, with a constant depending only on $\mathcal{L}$,
$R_1$ and $R_2$. Since $f$ vanishes when a coordinate has modulus in $(R_1,R_2)$, one deduces:
$$ \int_{\mathbb{C}^q} f \, d \left[\bigotimes_{k \geq 1,r \geq 1}
 \left( \mathcal{L}_k^{[r]} \right)^{\otimes \lambda_{k,r}} \right]
\leq C \, \prod_{k \geq 1, r \geq 1} \left( \frac{k(k-1)...(k-r+1)}{k^{q+1}} \right)^{\lambda_{k,r}}$$
where $C$ can depend on $\mathcal{L}$, $R_1$, $R_2$, $f$, $q$. Hence, one only needs to check that:
$$\sum_{(\lambda_{k,r})_{k, r \geq 1} \in \tilde{L}(q)} 
 \prod_{k, r \geq 1} \left[\frac{1}{\lambda_{k,r}!} \left( \frac{\theta C_k^r}{k^{\beta}} \right)^{\lambda_{k,r}}
 \right]   < \infty,$$
where $\beta > q+1$ and where by convention, the binomial coefficient $C_k^r$ is equal to zero for $k < r$.
Now, for $(\lambda_{k,r})_{k, r \geq 1} \in \tilde{L}(q)$:
  $$C_k^r \leq k^r$$ and 
$$\theta^{\sum_{k,r \geq 1} \lambda_{k,r}} \leq (1+\theta)^q.$$  
Hence it is sufficient to have
$$\sum_{(\lambda_{k,r})_{k, r \geq 1} \in \tilde{L}(q)} 
  \prod_{k, r \geq 1} k^{ \lambda_{k,r} (r- \beta)} < \infty.$$
Since
$$\sum_{k, r \geq 1} r \lambda_{k,r} = q,$$
one  necessarily has $\lambda_{k,r} = 0$ for $r > q$. Hence, in any case, $$k^{\lambda_{k,r} (r-\beta)}
\leq k^{\lambda_{k,r} \gamma}$$
for $\gamma = q - \beta < -1$, and Proposition
\ref{convmutildeq} is proved if
$$\sum_{(\lambda_{k,r})_{k, r \geq 1} \in \tilde{L}(q)} 
  \prod_{k, r \geq 1} k^{  \gamma \lambda_{k,r}} < \infty.$$
This last estimate is easy to check, by replacing $\tilde{L} (q)$ by the (larger) set of 
families $(\lambda_{k,r})_{k, r \geq 1}$ of nonnegative integers such that $\lambda_{k,r} = 0$ for
$r > q$ and $\lambda_{k,r} \leq q$ in any case. 
\end{proof}
\section{The point process of eigenangles and its scaling limit in the unitary case}

In this section we do a more precise study of the point process of eigenvalues, in the unitary case. In particular,
we obtain a scaling limit for the eigenangles if they are properly renormalized. This limit 
is dominated by the large cycles of permutations, hence, we can expect almost sure convergence results 
if we consider virtual permutations. More precisely, let $(M_N)_{N \geq 1}$ be a 
sequence of random matrices following the distribution $\mathbb{P} (\infty, \theta, \mathcal{L})$, 
where $\theta > 0$ and $\mathcal{L}$ is a probability distribution on the unit circle. 
Recall that the measure $\mu(M_N)$, representing the point process of eigenvalues can be written as:
$$\mu(M_N) = \sum_{m=1}^{\infty} \, \mathds{1}_{l_{N,m} > 0} \sum_{\omega^{l_{N,m}} = Z_{N,m}}
 \delta_{\omega}. $$
Here $l_{N,m}$ is the cardinality of the intersection $C_{N,m}$ of $C_m$ 
and $\{1,...,N\}$, and $$Z_{N,m} = \prod_{j \in C_{N,m}} z_j,$$
where $(C_m)_{m \geq 1}$ is the partition of $\mathbb{N}^*$ (ordered by increasing 
smallest element) given by the cycle structure of a virtual 
permutation following the Ewens($\theta$) measure, and $(z_j)_{j \geq 1}$ is an independent sequence of 
i.i.d. variables with law $\mathcal{L}$. Since all the eigenvalues are on the unit circle, 
it can be more practical to consider the eigenangles. The corresponding point process can be described as follows:
one takes the point $x$ if and only if $e^{ix}$ is an eigenvalue of $M_N$, with a multiplicity equal to the multiplicity 
of the eigenvalue. Note that this process is $2 \pi$-periodic, and the corresponding random measure $\tau(M_N)$ 
can be written as:
$$\tau(M_N) = \sum_{m=1}^{\infty} \, \mathds{1}_{l_{N,m} > 0} \sum_{x \equiv \Theta_{N,m} ( \operatorname{mod.}
\, 2 \pi/ l_{N,m} )} \delta_x,$$
where
$$\Theta_{N,m} = \frac{1}{ i l_{N,m}} \, \sum_{j \in C_{N,m}} \log(z_j)$$
is real and well-defined modulo $2 \pi/l_{N,m}$, for $l_{N,m} > 0$. 
Note that the construction of the measure $\tau(M_N)$ implies immediately that
 $\tau(M_N)( [0, 2 \pi)) = N$, in other words, the average
spacing of two consecutive points of the corresponding point process is equal to $2 \pi/N$. If 
we want to expect a convergence for $N$ going to infinity, we need to rescale $\tau(M_N)$ in order 
to have a constant average spacing, say, one. That is why we introduce the rescaled measure $\tau_N(M_N)$,
defined as the image of $\tau(M_N)$ by the multiplication by $N/2\pi$: the corresponding point process
contains the points $x$ such that $e^{2 i \pi x/N}$ is an eigenvalue of $M_N$. One checks that
$$\tau_N(M_N) = \sum_{m=1}^{\infty} \, \mathds{1}_{y_{N,m} > 0}  \sum_{k \in \mathbb{Z}}
 \delta_{(\gamma_{N,m} + k)/y_{N,m}},$$
where $y_{N,m} := l_{N,m}/N$ is the $m$-th renormalized cycle length of the permutation $\sigma_N$ associated
with $M_N$, and:
$$\gamma_{N,m} := \frac{1} {2 i \pi} \, \sum_{j \in C_{N,m}} \log(z_j)$$
is well-defined modulo $1$.
Now the general results on virtual permutations (see \cite{tsi1} or Section \ref{sec::2}) imply that for all $m \geq 1$, $y_{N,m}$
converges a.s. to a random variable $y_m$, where $(y_m)_{m \geq 1}$ is a GEM process of parameter 
$\theta$. Then, if $\gamma_{N,m}$ is supposed to be equal to zero, one can expect that the measure 
$\tau_N(M_N)$ converges to
$$\tau_{\infty} ((M_N)_{N \geq 1}) := \sum_{m=1}^{\infty} \, \sum_{k \in \mathbb{Z}} \delta_{k/y_m}.$$
Of course, one needs to be careful since $\tau_{\infty} ((M_N)_{N \geq 1})$ has an infinite Dirac mass
at zero. Moreover, the condition $\gamma_{N,m} = 0$ is satisfied if and only if $\mathcal{L}$ is $\delta_1$, the Dirac measure 
at  one. That is why we state the following result:
\begin{proposition} \label{conv1}
Let $(M_N)_{N \geq 1}$ be a sequence of random matrices, which follows the distribution $\mathbb{P} (\infty, 
\theta, \delta_1)$, for $\theta > 0$. Then, with the notation above, the random measure
 $\tau_{N} (M_N)$ converges a.s. to the random measure $\tau_{\infty} ((M_N)_{N \geq 1})$, in the following sense:
for all continuous functions $f$ from $\mathbb{R}$ to $\mathbb{R}_+$, with compact support,
$$\int_{\mathbb{R}} f \, d \tau_{N} (M_N)  \underset{N \rightarrow \infty}{\longrightarrow} \int_{\mathbb{R}}
 f \, d \tau_{\infty} ((M_N)_{N \geq 1})$$
almost surely.
\end{proposition}
\begin{remark}
In Proposition \ref{conv1}, the matrix $M_N$ is simply a permutation matrix, associated to a permutation 
$\sigma_N \in \Sigma_N$. Moreover, the positivity of $f$ is needed in order to deal with the infinite Dirac mass at 
zero of $\tau_{\infty} ((M_N)_{N \geq 1})$.
\end{remark}
\begin{proof}
If $f(0) > 0$, the integral of $f$ with respect to $\tau_{N} (M_N)$ is greater than or equal to 
$f(0)$ times the number of $m$ such that $y_{N,m} > 0$, i.e. the number of cycles of $\sigma_N$. 
Since this number tends a.s. to infinity with $N$, we are done. Then, we can suppose $f(0)=0$, which implies:
$$\int_{\mathbb{R}} f \, d \tau_{N} (M_N) =  \sum_{m=1}^{\infty} 
 \sum_{k \in \mathbb{Z} \backslash \{0\}}  f(k/y_{N,m}) $$
and 
$$\int_{\mathbb{R}} f \, d \tau_{\infty} ((M_N)_{N \geq 1}) =  \sum_{m=1}^{\infty} 
 \sum_{k \in \mathbb{Z} \backslash \{0\}}   f(k/y_{m}) $$
where we take the convention $f(k/y_{N,m})= 0$ for $y_{N,m} =0$. Now, if the support 
of $f$ is included in $[-A,A]$ and if $s_m$ is the supremum
of $y_{N,m}$ for all $N \geq 1$, it is clear that $f(k/y_{N,m}) = f(k/y_m) = 0$ for all $k, m$ such that
$|k| > A s_m$. One deduces that
$$\int_{\mathbb{R}} f \, d \tau_{N} (M_N) =  \sum_{(m,k) \in S}  f(k/y_{N,m}) $$
and 
$$\int_{\mathbb{R}} f \, d \tau_{\infty} ((M_N)_{N \geq 1}) =  \sum_{(m,k) \in S}  f(k/y_{m}) $$
where $S$ is the set (independent of $N$) of couples of integers $(m,k)$ such that $m \geq 1$ and $0 
< |k| \leq A s_m$. Now 
$$\sum_{m \geq 1} s_m < \infty$$
almost surely (see \eqref{2}), and then $S$ is a.s. a finite set. Since $f(k/y_{N,m})$ tends a.s. to 
$f(k/y_m)$ for all $k,m$ (recall that $f$ is continuous), we are done. 
\end{proof}
\noindent
The a.s. weak convergence given in Proposition \ref{conv1} cannot be directly generalized if the
 law $\mathcal{L}$ is not the Dirac measure at one, because $\gamma_{N,m}$ changes in a non-negligible way
for each $N \in C_m$. However, one can expect a weaker convergence. 
A natural candidate for the corresponding limit distribution would be the law of a random measure 
 defined by:
\begin{equation} 
 \sum_{m=1}^{\infty}  \sum_{k \in \mathbb{Z}} \delta_{(k + \chi_m)/y_m}, \label{chi}
\end{equation}
where we recall that $(y_m)_{m \geq 1}$ is a GEM($\theta$) process, and where $(\chi_m)_{m \geq 1}$
 is an independent i.i.d. sequence in $\mathbb{R} / \mathbb{Z}$. 
 The distribution
of $\chi_1$
needs to be, in a sense which has to be made precise, close to the distribution of:
$$\frac{1}{2 i \pi} \sum_{j=1}^K \log(z_j),$$
where $K$ is a large integer (recall that $(z_j)_{j \geq 1}$ is an i.i.d. sequence of random variables
with law $\mathcal{L}$). From the following result, we deduce a good candidate: 
\begin{lemma} \label{L*}
Let $\mathcal{L}$ be a probability measure on $\mathbb{U}$, and for $k \geq 1$, let $\mathcal{L}^{*k}$ be the multiplicative  
convolution of $k$ copies of $\mathcal{L}$. Moreover, let $r(\mathcal{L})$ be the infimum of the
integers $r$ such that $\mathcal{L}$ is carried by the set $\mathbb{U}_r$ of $r$-th roots of unity. Then for all 
sequences $(d_k)_{k \geq 1}$ of strictly positive integers, tending to infinity with $k$, the 
probability measure:
$$\frac{1} {d_k} \, \sum_{p=k}^{k+d_k-1} \mathcal{L}^{*p}$$
converges weakly to $\mathcal{L}^*$, where $\mathcal{L}^*$ is the uniform distribution on $\mathbb{U}_r$ if
 $r(\mathcal{L}) < \infty$, and the uniform distribution on $\mathbb{U}$ if $r(\mathcal{L}) = \infty$. 
\end{lemma}
\begin{proof}
Let us define, for all $q \in \mathbb{Z}$, and for all probability measures $\mathcal{M}$ on $\mathbb{U}$:
$$\widehat{\mathcal{M}} (q) := \int_{\mathbb{U}} z^q \, d \mathcal{M}.$$
Moreover, let us set, for $k \geq 1$:
$$\mathcal{M}_k := \frac{1} {d_k} \, \sum_{p=k}^{k+d_k-1} \mathcal{L}^{*p}$$
Then, one has:
$$\widehat{\mathcal{M}_k} (q) = \frac{1}{d_k} \, \sum_{p=k}^{k + d_k - 1} \, [\widehat{\mathcal{L}}(q)]^p,$$
which is equal to one if $\widehat{\mathcal{L}}(q) = 1$ and which tends to zero when $k \rightarrow \infty$, 
if $\widehat{\mathcal{L}}(q) \neq 1$. Now, $\widehat{\mathcal{L}}(q) = 1$ if and only if $q$ is divisible by $r
(\mathcal{L}) $, for $r(\mathcal{L})  < \infty$, 
and if and only if $q = 0$, for $r(\mathcal{L}) = \infty$. 
Hence:
$$ \widehat{\mathcal{L}^*} (q) = \mathds{1}_{\widehat{\mathcal{L}}(q) = 0},$$
and finally,
$$\widehat{\mathcal{M}_k} (q) \underset{k \rightarrow \infty}{\longrightarrow} \hat{\mathcal{L}^*} (q).$$
\end{proof}
\noindent
Because of Lemma \ref{L*}, the law $D(\mathcal{L})$ of $\chi_1$, in equation \eqref{chi}, is chosen as follows:
if $r(\mathcal{L}) < \infty$, it is the uniform distribution on the classes $\{0,1/r,...,(r-1)/r\}$ modulo 1, 
and if $r(\mathcal{L}) = \infty$, it is the uniform measure on $\mathbb{R} / \mathbb{Z}$. 
\begin{remark}
If $r(\mathcal{L})$ is finite, it is the smallest integer $r \geq 1$ such that $M_N \in \mathcal{H}_r(N)$ a.s., 
for $M_N$ following the distribution $\mathcal{P} (N, \theta, \mathcal{L})$. 
\end{remark}
\noindent
We are now able to state our convergence result for any distribution $\mathcal{L}$ on $\mathbb{U}$:
\begin{proposition} \label{conv2}
Let $(M_N)_{N \geq 1}$ be a sequence of random matrices, such that for all $N \geq 1$, 
$M_N$ follows the distribution $\mathbb{P} (N, 
\theta, \mathcal{L})$, where $\theta > 0$ and where $\mathcal{L}$ is a probability measure on $\mathbb{U}$. 
We suppose that $\mathcal{L}$ satisfies one of the two following conditions:
\begin{itemize}
\item the measure is carried by $\mathbb{U}_r$ for some integer $r \geq 1$. 
\item there exists $v > 1$ such that for $\epsilon > 0$ small enough, and for all
arcs $\mathcal{A}$ in $\mathbb{U}$ of size $\epsilon$, $\mathcal{L} (\mathcal{A}) \leq |\log(\epsilon)|^{-v}$.
\end{itemize}
\noindent
Let $\tau_{\infty} (\theta, \mathcal{L})$ be the random measure defined by:
$$ \tau_{\infty} (\theta, \mathcal{L}) := 
  \sum_{m=1}^{\infty}  \sum_{k \in \mathbb{Z}} \delta_{(k + \chi_m)/x_m},$$
where $(x_m)_{m \geq 1}$ is a Poisson-Dirichlet process of parameter $\theta$, and 
$(\chi_m)_{m \geq 1}$ is an independent sequence of i.i.d. random variables on $\mathbb{R} / \mathbb{Z}$, 
following the distribution $D(\mathcal{L})$ defined above. 
Then, 
with the previous notation, the distribution of the random measure
 $\tau_{N} (M_N)$ converges to the distribution of  the random measure $\tau_{\infty} (\theta, \mathcal{L})$, 
in the following sense:
for all continuous functions $f$ from $\mathbb{R}$ to $\mathbb{R}_+$, with compact support,
\begin{itemize}
\item if $f(0) > 0$ and $r(\mathcal{L}) < \infty$, then 
$$\int_{\mathbb{R}} f \, d \tau_{\infty} (\theta, \mathcal{L}) = \infty$$
a.s., and for all $A > 0$:
$$\mathbb{P} \left[ \int_{\mathbb{R}} f \, d \tau_{N} (M_N) \leq A \right] \longrightarrow 0$$ 
when $N \rightarrow \infty$;
\item if $f(0)= 0$ or $r(\mathcal{L}) = \infty$, then:
$$\int_{\mathbb{R}} f \, d \tau_{\infty} (\theta, \mathcal{L}) < \infty$$
a.s., and 
$$\int_{\mathbb{R}} f \, d \tau_{N} (M_N) \underset{N \rightarrow \infty}{\longrightarrow} 
 \int_{\mathbb{R}} f \, d \tau_{\infty} (\theta, \mathcal{L}) $$
in distribution.
\end{itemize}
\end{proposition}
\noindent
\begin{remark}
In Proposition \ref{conv2}, the Poisson-Dirichlet distribution can be replaced by a GEM distribution, 
since it does not change the law of the random measure $\tau_{\infty} (\theta, \mathcal{L})$.
\end{remark}
\begin{proof}
 Without changing the laws of the random measures $\tau_{N} (M_N)$ and 
 $\tau_{\infty} (\theta, \mathcal{L})$,  
  one can suppose that $(M_N)_{N \geq 1}$ follows
the distribution $\mathbb{P}(\infty, \theta, \mathcal{L})$ and that, with the notation above:
$$\tau_{\infty} (\theta, \mathcal{L}) =
  \sum_{m=1}^{\infty}  \sum_{k \in \mathbb{Z}} \delta_{(k + \chi_m)/y_m},$$
where $(\chi_m)_{m \geq 1}$ is supposed to be independent of $(y_m)_{m \geq 1}$. 
Recall that 
$$\tau_N(M_N) = \sum_{m=1}^{\infty} \, \mathds{1}_{y_{N,m} > 0}  \sum_{k \in \mathbb{Z}} 
 \delta_{(k + \gamma_{N,m})/y_{N,m}},$$
where, conditionally on $(y_{N,m})_{m \geq 1}$,
 $(\gamma_{N,m})_{m \geq 1}$ is a sequence of independent random variables on $\mathbb{R} / \mathbb{Z}$, 
and for all $m \geq 1$, $\gamma_{N,m}$ has the same law as $\frac{\log Z}{2 i \pi}$, where $Z$ follows
the distribution $\mathcal{L}^{*Ny_{N,m}} = \mathcal{L}^{*l_{N,m}}$. 
Let us now suppose $f(0)> 0$ and $r(\mathcal{L}) < \infty$. Since $\chi_m$ is a.s. equal to zero 
for infinitely many $m$ (because $D(\mathcal{L}) (\{0\}) > 0$), $\tau_{\infty} (\theta, \mathcal{L})$
has an infinite Dirac mass at zero, and then a.s.:
$$\int_{\mathbb{R}} f \, d \tau_{\infty} (\theta, \mathcal{L}) = \infty.$$
 Moreover
$$\int_{\mathbb{R}} f \, d \tau_{N} (M_N)  \geq f(0) \, \sum_{m=1}^{\infty} \mathds{1}_{y_{N,m} > 0,
\gamma_{N,m} = 0}.$$
Now, if $(Z_{p,q})_{p, q \geq 1}$ is a family of independent random variables, independent
of $\sigma_N$ (the permutation associated with $M_N$), and  such that for all $p,q \geq 1$, $Z_{p,q}$ has
distribution $\mathcal{L}^{*p}$, then
$$\sum_{m=1}^{\infty} \mathds{1}_{y_{N,m} > 0,
\gamma_{N,m} = 0} = \sum_{p=1}^{\infty} \sum_{q=1}^{a_{N,p}} \mathds{1}_{Z_{p,q} = 1}$$
in distribution, where $a_{N,p}$ denotes the number of $m$ such that $l_{N,m} = p$ (i.e. the number of 
$p$-cycles in $\sigma_N$). Therefore:
$$\int_{\mathbb{R}} f \, d \tau_{N} (M_N)  \succeq  f(0) \sum_{p=1}^{\infty}
 \sum_{q=1}^{c_{N,p}} \mathds{1}_{Z_{p,q} = 1},$$
where $\succeq$ denotes the stochastic domination, and where $(c_{N,p})_{p \geq 1}$, defined
 by the Feller coupling as in \eqref{cN}, is supposed to be independent of $(Z_{p,q})_{p, q \geq 1}$.
Since for all $p \geq 1$, $c_{N,p}$ increases a.s. to $b_p$, Proposition \ref{conv2} is proved 
for $f(0)> 0$ and $r(\mathcal{L}) < \infty$, if we show that a.s.,
$$\sum_{p=1}^{\infty} \sum_{q=1}^{b_p} \mathds{1}_{Z_{p,q} = 1} = \infty.$$
and a fortiori, if we prove that
$$\sum_{p=1}^{\infty} \mathds{1}_{ \exists q \leq b_p, Z_{p,q} = 1}  = \infty.$$
Since all the variables $(b_p)_{p \geq 1}$ and $(Z_{p,q})_{p, q \geq 1}$ are independent, by the Borel-Cantelli
lemma it is sufficient to have
$$\sum_{p=1}^{\infty} \mathbb{P} [ \exists q \leq b_p, Z_{p,q} = 1]  = \infty.$$
Now, if for $p \geq 1$, we define
 $$P(p) = \mathcal{L}^{*p} (\{1\}),$$
we have:
$$\mathbb{P} [\exists q \leq b_p, Z_{p,q} =1  \, | \, b_p] = 1 - [1 - P(p)]^{b_p},$$
and, since $b_p$ is a Poisson random variable of parameter $\theta/p$,
$$\mathbb{P} [\exists q \leq b_p, Z_{p,q} =1 ] = 1 - e^{- \theta P(p)/p},$$
and then we are done if we prove:
\begin{equation}
\sum_{p \geq 1} P(p)/p = \infty. \label{xxxxx}
\end{equation}
Now, by Lemma \ref{L*}:
$$\frac{1} {2^k} \, \sum_{p=2^k}^{2^{k+1}-1} \mathcal{L}^{*p}$$
converges weakly to the uniform distribution on $\mathbb{U}_{r(\mathcal{L})}$, for $k \rightarrow \infty$. 
 Since all the measures involved here are supported by
 the finite set $\mathbb{U}_{r(\mathcal{L})}$, one deduces that
$$\frac{1}{2^k} \, \sum_{p=2^k}^{2^{k+1}-1} P(p) \underset{k \rightarrow \infty}{\longrightarrow}
1/ r(\mathcal{L}),$$
which implies \eqref{xxxxx}. We can now suppose $f(0)= 0$ or $r(\mathcal{L}) = \infty$.
In order to prove Proposition \ref{conv2} in this case, we need the following result:
\begin{lemma} \label{Psi}
Let $\theta > 0$, $r \in \mathbb{N}^*$, and let $(l_{N,m})_{m \geq 1}$ be the sequence of cycle lengths (ordered 
by increasing smallest elements, and completed by zeros) of a random permutation in $\Sigma_N$, following the Ewens 
measure with parameter $\theta$. Then, there exists a function $\Psi_{r, \theta}$ from $\mathbb{N}^*$ 
to $\mathbb{R}^+$, tending to zero at infinity, and such that for all strictly positive integers 
 $(l_m)_{1 \leq m \leq r}$, $(l'_m)_{1 \leq m \leq r}$, $N$,  satisfying:
$$\sum_{m=1}^{r} l_m < N- N^{2/3},$$
$$\sum_{m=1}^{r} l'_m < N - N^{2/3},$$
$$\forall m \in \{1,...,r\}, \, \exists q \in \mathbb{N}^*, \, q^2 \leq l_m, l'_m < (q+1)^2,$$
one has
$$\frac{\mathbb{P} [\forall m \in \{1,...,r\}, l_{N,m} = l_m]}
{\mathbb{P} [\forall m \in \{1,...,r\}, l_{N,m} = l'_m]} \in \left( 1 - \Psi_{r, \theta} (N), 
1 + \Psi_{r, \theta} (N) \right).$$
\end{lemma}
\begin{proof}
For $0 \leq m \leq r$, let us define:
$$L_m = \sum_{p = 1}^m l_p,$$
$$L'_m = \sum_{p = 1}^m l'_p,$$
By using the Feller coupling, one obtains the following expression:
\begin{align*}
\mathbb{P} [\forall m \in \{1,...,r\}, l_{N,m} = l_m]
& = \prod_{m=1}^r \, \left[ \frac{\theta}{N - L_m + \theta} \, 
\prod_{p = L_{m-1}+1}^{L_m -1} \, \frac{N-p}{N-p+ \theta} \right] \\
& = \left(\prod_{p=1}^{L_r} \, \frac{N-p}{N-p+\theta} \right) \left( \prod_{m=1}^r \, \frac{\theta}{N-L_m} \right).
\end{align*}
\noindent
Hence, if $L_r \geq L'_r$
$$\frac{\mathbb{P} [\forall m \in \{1,...,r\}, l_{N,m} = l_m]}
{\mathbb{P} [\forall m \in \{1,...,r\}, l_{N,m} = l'_m]} 
= \left(\prod_{p=L'_r + 1}^{L_r} \, \frac{N-p}{N-p+\theta} \right) \left( \prod_{m=1}^r \, \frac{N-L'_m}{N-L_m}
\right).$$
For $1 \leq m \leq r$, there exists an integer $q \geq 1$ such that $q^2 \leq l_m, l'_m < (q+1)^2$, and
 $q \leq \sqrt{N}$, since $l_m, l'_m < N$.
Hence, $|l_m - l'_m| \leq 2q \leq 2 \sqrt{N}$, and $|L_m - L'_m| \leq 2r \sqrt{N}$. 
Moreover, $N- L_m, N-L'_m > N^{2/3}$ by assumption. Therefore, one has the majorization:
$$\left| \log \left(  \frac{\mathbb{P} [\forall m \in \{1,...,r\}, l_{N,m} = l_m]}
{\mathbb{P} [\forall m \in \{1,...,r\}, l_{N,m} = l'_m]} 
\right) \right|  \leq  2r \sqrt{N} \, \log \left( 1 + \frac{\theta}{N^{2/3}} \right) + r \log \left(1+ \frac{
2 r \sqrt{N}}{N^{2/3}} \right).$$
This implies Lemma \ref{Psi} for:
$$\Psi_{r, \theta} (N) = \left(1 + \frac{\theta}{N^{2/3}} \right)^{2r \sqrt{N}} \left(1 + \frac{2 r \sqrt{N}}
{N^{2/3}} \right)^r + \frac{1}{N} - 1.$$
\end{proof}
A consequence of Lemma \ref{Psi} is the following:
\begin{lemma} \label{ygamma}
Let $\theta > 0$ and let $\mathcal{L}$ be a probability law on the unit circle. 
Let $(l_{N,m})_{m \geq 1}$ be the sequence of cycle lengths (ordered 
by increasing smallest elements, and completed by zeros) of a random permutation in $\Sigma_N$, following Ewens 
measure with parameter $\theta > 0$. Let $(\gamma_{N,m})_{m \geq 1}$ be a sequence of random variables 
on $\mathbb{R} / \mathbb{Z}$ such that conditionally on $(l_{N,m})_{m \geq 1}$, 
 $(\gamma_{N,m})_{m \geq 1}$ are independent
and for all $m \geq 1$, $\gamma_{N,m}$ has the same law as $\frac{\log Z}{2 i \pi}$, where $Z$ follows
the multiplicative convolution $\mathcal{L}^{*l_{N,m}}$ of $l_{N,m}$ copies of $\mathcal{L}$.
Then, for all $r \geq 1$, the family $(l_{N,m}/N, \gamma_{N,m})_{m \leq r}$ converges in distribution 
to $(y_m, \chi_{m})_{m \leq r}$, where $(y_m)_{m \geq 1}$ is a GEM process of parameter $\theta$, and
$(\chi_m)_{m \geq 1}$ is an independent sequence of i.i.d. random variables, with law $D(\mathcal{L})$, defined
above.
\end{lemma}
\begin{proof}
Let $q_{N,m}$ be the integer part of the square root of $l_{N,m}$. For all continuous functions
$g$ from $(\mathbb{R} / \mathbb{Z})^r$ to $\mathbb{R}_+$, one has:
$$\mathbb{E} [g((\gamma_{N,m})_{m \leq r} ) | (q_{N,m})_{m \leq r}] 
 =\int_{\mathbb{U}^r} g(\log z_1/2 i \pi, \dots, \log z_r/2 i \pi) \, d \mathcal{N}_{(q_{N,m})_{m \leq r}}(z_1,\dots,z_r),$$
where $\mathcal{N}_{(q_{N,m})_{m \leq r}} $ is defined as the conditional expectation of a random measure, more precisely:
$$\mathcal{N}_{(q_{N,m})_{m \leq r}} = \mathbb{E} \left[ \left. \bigotimes_{m = 1}^r \mathcal{L}^{*l_{N,m}} 
 \, \right| \, (q_{N,m})_{m \leq r} \right].$$
Let us suppose that $q_{N,m} \geq N^{1/3}$ for all $m \leq r$ and:
$$N - \sum_{m =1}^r (q_{N,m} + 1)^2 > N^{2/3}.$$
By Lemma \ref{Psi}, the conditional law of $(l_{N,m})_{m \leq r}$ given $(q_{N,m})_{m \leq r}$ is 
a linear combination of Dirac measures on the $r$-uples $(l_m)_{m \leq r}$ of integers such that 
$q_{N,m}^2 \leq l_m < (q_{N,m}+1)^2$, and the quotient between two coefficients of this combination lies on 
the interval $(1- \Psi_{r, \theta} (N), 1 + \Psi_{r, \theta} (N))$.
One deduces that
$$\left(1 - \Psi_{r, \theta} (N) \right) \, \bigotimes_{m=1}^r
\mathcal{L}'_{q_{N,m}} \leq \mathcal{N}_{(q_{N,m})_{m \leq r}} 
\leq \left(1 + \Psi_{r, \theta} (N) \right) \, \bigotimes_{m=1}^r
\mathcal{L}'_{q_{N,m}} $$
where, for $q \geq 1$:
$$\mathcal{L}'_{q} = \frac{1}{2q+1} \, \sum_{l=q^2}^{q^2 + 2q} \, \mathcal{L}^{*l}.$$
Hence:
$$\mathbb{E} [g((\gamma_{N,m})_{m \leq r} ) | (q_{N,m})_{m \leq r}] 
 = C \, \int_{\mathbb{U}^r} g(\log z_1 / 2i \pi, \dots, \log z_r/2 i \pi) \, \prod_{m=1}^r \mathcal{L}'_{q_{N,m}} (dz_m) $$
for
$$1 - \Psi_{r, \theta} (N) \leq C \leq 1 + \Psi_{r, \theta} (N).$$
Now, by Lemma \ref{L*}, $\mathcal{L}'_{q}$ converges to $\mathcal{L}^*$ when $q$ goes to infinity, hence,
there exists a function $K_g$ from $\mathbb{R}_+$ to $\mathbb{R}_+$, tending to zero at infinity, and such
that for all $q \geq q_0 \geq 1$, $q$ integer:
\begin{align*}
& \left| 
\int_{\mathbb{U}^r} g(\log z_1 / 2 i \pi, \dots, \log z_r/ 2 i \pi) \, \prod_{m=1}^r \mathcal{L}'_{q_{N,m}} (dz_m) \right. \\ 
-  & \left. \int_{\mathbb{U}^r} g(\log z_1 / 2 i \pi, \dots, \log z_r/ 2 i \pi) \, \prod_{m=1}^r \mathcal{L}^* (dz_m) \right| 
\leq K_g(q_0).
\end{align*}
Therefore, if $N$ is large enough, if $q_{N,m} \geq N^{1/3}$ for all $m \leq r$ and if
$$N - \sum_{m =1}^r (q_{N_m} + 1)^2 > N^{2/3},$$
then:
$$\mathbb{E} [g((\gamma_{N,m})_{m \leq r} ) | (q_{N,m})_{m \leq r}] 
 \in [A_1,B_1],$$
for 
$$A_1 = \left(1 - \Psi_{r, \theta} (N)\right) \,   \left[- K_g(N^{1/3}) + 
 \int_{\mathbb{U}^r} g(\log z_1 / 2 i \pi, \dots, \log z_r/ 2 i \pi) \, \prod_{m=1}^r D(\mathcal{L}) (dz_m) \right] $$
and 
$$B_1 = \left(1 + \Psi_{r, \theta} (N)\right) \,   \left[ K_g(N^{1/3}) + 
\int_{\mathbb{U}^r} g(\log z_1 / 2 i \pi, \dots, \log z_r/ 2 i \pi) \, \prod_{m=1}^r D(\mathcal{L}) (dz_m)\right]. $$
Now, let $f$ be a continuous function from $[0,1]^r$ to $\mathbb{R}_+$.
One has:
$$\mathbb{E} [f((q_{N,m}^2/N)_{m \leq r}) g((\gamma_{N,m})_{m \leq r} ) ]
\in [A_2,B_2],$$
for
$$A_2 = A_1 \mathbb{E} [f((q_{N,m}^2/N)_{m \leq r}) \mathds{1}_{\mathcal{E}} ]$$
and
$$B_2 = B_1 \mathbb{E} [f((q_{N,m}^2/N)_{m \leq r}) \mathds{1}_{\mathcal{E}} ]
+ ||f||_{\infty} ||g||_{\infty} \mathbb{P} [\mathcal{E}^c],$$
where $$\mathcal{E} := \{ \forall m \leq r, q_{N,m} \geq 
 N^{1/3}, N - \sum_{m =1}^r (q_{N,m} + 1)^2 > N^{2/3} \}$$
and $\mathcal{E}^c$ is the complement of the event $\mathcal{E}$. 
If we suppose that $(l_{N,m})_{N,m \geq 1}$ are the cycle lengths associated with a virtual permutation
following the Ewens($\theta$) measure, then for all $m \leq r$, $q_{N,m}^2/N$ tends a.s. to 
$y_m$, for a GEM($\theta$) processes $(y_m)_{m \geq 1}$, and the event $\mathcal{E}^c$ holds for finitely
many values of $N$. Therefore:
$$ \mathbb{E} [f((q_{N,m}^2/N)_{m \leq r}) \mathds{1}_{\mathcal{E}} ]
\underset{N \rightarrow \infty}{\longrightarrow} \mathbb{E} [f((y_m)_{m \leq r})]$$
and  $$\mathbb{P} [\mathcal{E}^c] \underset{N \rightarrow \infty}{\longrightarrow} 0. $$
One deduces that
$$\mathbb{E} [f((q_{N,m}^2/N)_{m \leq r}) g((\gamma_{N,m})_{m \leq r} ) ]
\underset{N \rightarrow \infty}{\longrightarrow} 
\mathbb{E} [f((y_m)_{m \leq r}) g((\chi_m)_{m \leq r})]$$
(recall that $(y_m)_{m \leq r}$ and $(\chi_m)_{m \leq r}$ are supposed to be independent). 
Now, $$\left| \frac{q_{N,m}^2}{N} - \frac{l_{N,m}}{N} \right| \leq \frac{2}{\sqrt{N}},$$
hence, by the uniform continuity of $f$ and the boundedness of $g$:
$$\mathbb{E} [f(l_{N,m}/N)_{m \leq r}) g((\gamma_{N,m})_{m \leq r} ) ]
\underset{N \rightarrow \infty}{\longrightarrow} 
\mathbb{E} [f((y_m)_{m \leq r}) g((\chi_m)_{m \leq r})],$$
which proves Lemma \ref{ygamma}.
\end{proof}
Now, let $f$ be a continuous, nonnegative function with compact support, and 
$\Phi$ a continuous function from $\mathbb{R}_+$ to $[0,1]$, such that 
$\Phi(x)=0$ for $x \leq 1$ and $\Phi(x) =1$ for $x \geq 2$. Let us suppose 
$f(0)= 0$ or $r(\mathcal{L}) = \infty$. We first remark that for all integers $r, s \geq 1$:
$$\sum_{m=1}^r \Phi(sy_{N,m}) \sum_{k \in \mathbb{Z}} f[(k + \gamma_{N,m})/y_{N,m}]
\underset{N \rightarrow \infty}{\longrightarrow} 
\sum_{m=1}^{r} \Phi(sy_{m})  \sum_{k \in \mathbb{Z}} f[(k + \chi_m)/y_m]$$
in distribution. This is a consequence of Lemma \ref{ygamma} and the fact that
$$\Phi(sy) \sum_{k \in \mathbb{Z}} f[(k + \chi)/y]$$
is continuous with respect to $(y,\chi) \in [0,1] \times (\mathbb{R} / \mathbb{Z})$. 
Now, if $f$ is supported by $[-A,A]$ (for $A > 0$), one has the inequalities:
$$\sum_{k \in \mathbb{Z}} f[(k + \gamma_{N,m})/y_{N,m}] \leq (2A +1) ||f||_{\infty}$$
and 
$$\sum_{k \in \mathbb{Z}} f[(k + \chi_m)/y_{m}] \leq (2A +1) ||f||_{\infty}.$$
Therefore, for $\lambda \in \mathbb{R}_+$:
\begin{align*}
& \left| \mathbb{E} \left[ e^{i \lambda \sum_{m=1}^r \mathds{1}_{y_{N,m} > 0}
 \sum_{k \in \mathbb{Z}} f[(k + \gamma_{N,m})/y_{N,m}]}\right] -
\mathbb{E} \left[ e^{i \lambda \sum_{m=1}^r 
 \sum_{k \in \mathbb{Z}} f[(k + \chi_m)/y_{m}]} \right] \right| 
 \\ \leq & \left| \mathbb{E} \left[ e^{i \lambda \sum_{m=1}^r \Phi(sy_{N,m})
 \sum_{k \in \mathbb{Z}} f[(k + \gamma_{N,m})/y_{N,m}]} \right] -
\mathbb{E} \left[ e^{i \lambda \sum_{m=1}^r \Phi(sy_m)
 \sum_{k \in \mathbb{Z}}  f[(k + \chi_m)/y_{m}]} \right] \right| 
\\ + & (2A+1)\lambda||f||_{\infty} \sum_{m=1}^r \left( \mathbb{P}[0 < y_{N,m} \leq 2/s] + \mathbb{P}
[y_m \leq 2/s] \right).
\end{align*}
\noindent
The first term in the right-hand side of this inequality tends to zero when $N$ goes to infinity. 
Hence:
\begin{align*}
& \underset{N \rightarrow \infty}{\operatorname{lim \, sup}} \, 
\left| \mathbb{E} \left[ e^{i \lambda \sum_{m=1}^r \mathds{1}_{y_{N,m} > 0}
 \sum_{k \in \mathbb{Z}} f[(k + \gamma_{N,m})/y_{N,m}]}\right] -
\mathbb{E} \left[ e^{i \lambda \sum_{m=1}^r 
 \sum_{k \in \mathbb{Z}} f[(k + \chi_m)/y_{m}]} \right] \right| \\ 
& \leq (2A+1)\lambda||f||_{\infty} \sum_{m=1}^r \left( 
\underset{N \rightarrow \infty}{\operatorname{lim \, sup}} \, \mathbb{P}[y_{N,m} \leq 2/s] + \mathbb{P}
[y_m \leq 2/s] \right).
\end{align*}
\noindent
Now if we assume that $(y_{N,m})_{N,m \geq 1}$ is the family of renormalized cycle lengths of 
a virtual permutation following the Ewens($\theta$) measure, we obtain, 
 by Fatou's lemma:
$$\underset{N \rightarrow \infty}{\operatorname{lim \, inf}} \, \mathbb{P} [y_{N,m} > 2/s]  \geq
\mathbb{E} \left[\underset{N \rightarrow \infty}{\operatorname{lim \, inf}}
\, \mathds{1}_{y_{N,m} > 2/s} \right].$$
Since $y_{N,m}$ tends to $y_m$ for $N$ going to infinity:
$$\underset{N \rightarrow \infty}{\operatorname{lim \, sup}} \, \mathbb{P} [y_{N,m} \leq 2/s]  
\leq \mathbb{P} [y_m \leq 2/s].$$
By taking $s \rightarrow \infty$, we deduce that:
\begin{equation}
\left| \mathbb{E} \left[ e^{i \lambda \sum_{m=1}^r \mathds{1}_{y_{N,m} > 0}
 \sum_{k \in \mathbb{Z}} f[(k + \gamma_{N,m})/y_{N,m}]}\right] -
\mathbb{E} \left[ e^{i \lambda \sum_{m=1}^r 
 \sum_{k \in \mathbb{Z}} f[(k + \chi_m)/y_{m}]} \right] \right| \label{Fourier}
\end{equation}
tends to zero when $N$ goes to infinity. 
Now, for all $\chi \in [0,1)$ and $y \in (0,1]$:
$$\sum_{k \in \mathbb{Z}} f[(k + \chi)/y] \leq (2A+1) ||f||_{\infty} \left( \mathds{1}_{\chi \leq Ay}
+ \mathds{1}_{1-\chi \leq Ay} \right)$$
and if $f(0) = 0$, $\chi = 0$:
$$\sum_{k \in \mathbb{Z}} f[(k + \chi)/y] \leq (2A+1) ||f||_{\infty} \mathds{1}_{Ay \geq 1}.$$
Therefore, if $\gamma_{N,m}$ is identified with the unique element of $[0,1)$ in its
congruence class modulo 1, then:
\begin{align}
& \mathbb{E} \left[ \sum_{m=r+1}^{\infty} \mathds{1}_{y_{N,m} > 0}
 \sum_{k \in \mathbb{Z}} f[(k + \gamma_{N,m})/y_{N,m}] \right] \nonumber \\
 & \leq (2A + 1) ||f||_{\infty} \sum_{m=r+1}^{\infty} 
\left[ \mathbb{P} (y_{N,m} > 0,  0 < \gamma_{N,m} \leq Ay_{N,m}) \right. \nonumber  \\
& \left. + \mathbb{P} (y_{N,m} > 0,  1 - \gamma_{N,m} \leq Ay_{N,m})
+ \mathbb{P} (\gamma_{N,m} = 0, Ay_{N,m} \geq 1) \right] \label{S}
\end{align}
if $f(0) = 0$. In fact, this inequality remains true for $f(0)>0$. Indeed, in this case, 
one has, by assumption, $r(\mathcal{L})= \infty$, and then by the conditions given in Proposition \ref{conv2}, 
$\mathcal{L}$ has no atom. Since
\begin{align*}
& \mathbb{E} \left[ \sum_{m=r+1}^{\infty} \mathds{1}_{y_{N,m} > 0}
 \sum_{k \in \mathbb{Z}} f[(k + \gamma_{N,m})/y_{N,m}] \right]
\\ \leq & (2A + 1) ||f||_{\infty} \sum_{m=r+1}^{\infty} 
\left[ \mathbb{P} ( y_{N,m} > 0,  \gamma_{N,m} \leq Ay_{N,m}) + \mathbb{P} (y_{N,m} > 0,
 1 - \gamma_{N,m} \leq Ay_{N,m}) \right],
\end{align*}
we also have  \eqref{S}. 
Now, for all $\rho \in (0,1)$:
$$\mathbb{P} [y_{N,m} > 0, 0 < \gamma_{N,m} \leq Ay_{N,m}] = 
\mathbb{P} [y_{N,m} > 0, 0 < \gamma_{N,m} \leq \rho^m] + \mathbb{P} [y_{N,m} \geq \rho^m/A].$$
Since conditionally on $\{y_{N,m}> 0\}$,  the law of $\gamma_{N,m}$ is a convex combination
 of iterated convolutions of $\mathcal{L}$, 
it satisfies the same assumptions as $\mathcal{L}$, with the same underlying constants.  
One deduces that there exist $v > 1$, $m_0 \geq 1$, $C > 0$, independent of
$N$, such that for $m \geq m_0$:
\begin{equation}
\mathbb{P} [y_{N,m} > 0, 0 < \gamma_{N,m} \leq \rho^m] \leq C m^{-v}. \label{gamma1}
\end{equation}
Now, if the normalized lengths of cycles $(y_{N,m})_{N, m \geq 1}$ are associated with a virtual permutation
following the Ewens($\theta$) measure, and if for all $m \geq 1$, $s_m$ is the supremum of $y_{N,m}$ for 
all $N \geq 1$, then the expectation of $s_m$ decreases exponentially with $m$ (this result is contained 
in our proof of \eqref{sm} above). One deduces that there exist $\rho \in (0,1)$
and $m_1 \geq 1$, independent of $N$ and such that for $m \geq m_0$:
\begin{equation}
\mathbb{P} [y_{N,m} \geq \rho^m/A] \leq \rho^m. \label{gamma2}
\end{equation}
From \eqref{gamma1} and \eqref{gamma2}, 
$$\sup_{N \geq 1} \sum_{m = r+1}^{\infty} \mathbb{P} [y_{N,m} > 0, 0 < \gamma_{N,m} \leq Ay_{N,m}] 
\underset{r \rightarrow \infty}{\longrightarrow} 0.$$
In a similar way, one can prove:
$$\sup_{N \geq 1} \sum_{m = r+1}^{\infty} \mathbb{P} [y_{N,m} > 0, 1 -  \gamma_{N,m} \leq Ay_{N,m}] 
\underset{r \rightarrow \infty}{\longrightarrow} 0$$
and
$$\sup_{N \geq 1} \sum_{m = r+1}^{\infty} \mathbb{P} [\gamma_{N,m} = 0, Ay_{N,m} \geq 1] 
\underset{r \rightarrow \infty}{\longrightarrow} 0.$$
Therefore, by \eqref{S}:
\begin{equation}
\sup_{N \geq 1} \mathbb{E} \left[ \sum_{m=r+1}^{\infty} \mathds{1}_{y_{N,m} > 0}
 \sum_{k \in \mathbb{Z}} f[(k + \gamma_{N,m})/y_{N,m}] \right] \underset{r \rightarrow \infty}
{\longrightarrow} 0. \label{S1}
\end{equation}
By replacing $\gamma_{N,m}$ by $\chi_m$ and $y_{N,m}$ by $y_m$, one obtains:
\begin{equation}
\mathbb{E} \left[ \sum_{m=r+1}^{\infty} \mathds{1}_{y_{m} > 0}
 \sum_{k \in \mathbb{Z}} f[(k + \chi_{m})/y_{m}] \right] \underset{r \rightarrow \infty}
{\longrightarrow} 0, \label{S2}
\end{equation}
since the law of $\chi_m$ satisfies the same assumptions as $\mathcal{L}$ and the expectation of 
$y_m$ decreases exponentially with $m$. Since the quantity given by \eqref{Fourier} tends to zero
when $N$ goes to infinity, one easily deduces from \eqref{S1} and \eqref{S2} that it is also the 
case for:
$$\left| \mathbb{E} \left[ e^{i \lambda \sum_{m=1}^{\infty} \mathds{1}_{y_{N,m} > 0}
 \sum_{k \in \mathbb{Z}} f[(k + \gamma_{N,m})/y_{N,m}]}\right] -
\mathbb{E} \left[ e^{i \lambda \sum_{m=1}^{\infty}
 \sum_{k \in \mathbb{Z}} f[(k + \chi_m)/y_{m}]} \right] \right|,$$ 
which proves Proposition \ref{conv2}.
\end{proof}
\section{The uniform case}
\subsection{Eigenvalues distributions and correlation measures}
In this section, we focus on the uniform case, i.e. the case where for $N \geq 1$, 
the random matrix $M_N$ follows the law $\mathbb{P} (N, \theta, \mathcal{L})$, where $\theta > 0$
 and $\mathcal{L}$ is the uniform distribution on the unit circle. Let $(y_{N,m})_{N,m \geq 1}$ 
be the family of renormalized cycle lengths corresponding to a virtual permutation following the
Ewens($\theta$) measure, $y_m$ the limit of $y_{N,m}$ for $N$ going to infinity, and $(\chi_m)_{m \geq 1}$
a sequence of i.i.d. uniform random variables on $[0,1)$, independent of $(y_{N,m})_{N,m \geq 1}$.
The random measure $\tau_N (M_N)$ has the same law as 
$$\bar{\tau}_{N} :=  \sum_{m=1}^{\infty} \mathds{1}_{y_{N,m} > 0} \, \sum_{k \in \mathbb{Z}}
\delta_{(\chi_m + k)/y_{N,m}}$$
and $\tau_{\infty} (\theta, \mathcal{L})$ is equal to 
$$\bar{\tau}_{\infty} := \sum_{m=1}^{\infty}
 \sum_{k \in \mathbb{Z}}
\delta_{(\chi_m + k)/y_{m}}$$
in distribution. This description of the law of $\tau_N (M_N)$ and $\tau_{\infty} (\theta, \mathcal{L})$ 
implies the following remarkable property:
\begin{proposition}
Under the assumptions given above, the distributions of the random measures $\tau_N (M_N)$ and $\tau_{\infty}
(\theta, \mathcal{L})$ are invariant by translation.
\end{proposition}
\begin{proof}
It is enough to prove it for $\bar{\tau}_{N}$ and $\bar{\tau}_{\infty}$. The image of $\bar{\tau}_N$ by
a translation of $A \in \mathbb{R}$ is equal to
$$\sum_{m=1}^{\infty} \mathds{1}_{y_{N,m} > 0} \, \sum_{k \in \mathbb{Z}} \delta_{(\chi'_m+k)/y_{N,m}},$$
where $\chi'_m$ is the fractional part of $\chi_m + A y_{N,m}$. Now one easily sees that
 conditionally on $(y_{N,m})_{m \geq 1}$, the sequence $(\chi'_m)_{m \geq 1}$ is an i.i.d. sequence of 
uniform variables on $[0,1)$, as $(\chi_m)_{m \geq 1}$. This implies the invariance by translation 
of $\bar{\tau}_{N}$. For $\bar{\tau}_{\infty}$, the proof is exactly similar. 
\end{proof}
\noindent
The main interest of the introduction of the measures $\bar{\tau}_N$ and $\bar{\tau}_{\infty}$ is the 
following: we cannot expect an a.s. convergence of the random measure
$\tau_N (M_N)$ when $N$ goes to infinity, even if $(M_N)_{N \geq 1}$ follows 
the distribution $\mathbb{P} (\infty, \theta, \mathcal{L})$ because, with the notation of
the previous section, the "shifts" $\gamma_{N,m}$ (which are equal to $\chi_m$ in distribution), do
not converge a.s. when $N$ goes to infinity; however, since we take the same variables $\chi_m$ in 
the definitions of $\bar{\tau}_{N}$ and $\bar{\tau}_{\infty}$, we can expect such a convergence 
for these new measures. More precisely, we have the following:
\begin{proposition} \label{bartau}
Almost surely, with the notation above, the random measure $\bar{\tau}_{N}$  converges vaguely to
$\bar{\tau}_{\infty}$, which is locally finite. 
\end{proposition}
\begin{proof}
The fact that $\bar{\tau}_{\infty}$ is locally finite is a consequence of Proposition \ref{conv2} (one 
has $r(\mathcal{L}) = \infty$). Now, let $f$ be a continuous function from $\mathbb{R}$ to $\mathbb{R}$, 
with support included in $[-A,A]$ (for some $A > 0$). One has:
$$\int_{\mathbb{R}} f \, d\bar{\tau}_{N} =
  \sum_{m=1}^{\infty}  \sum_{k \in \mathbb{Z}} \mathds{1}_{y_{N,m} > 0} f((\chi_m + k)/y_{N,m})$$
and
$$\int_{\mathbb{R}} f \, d\bar{\tau}_{\infty} =
  \sum_{m=1}^{\infty}  \sum_{k \in \mathbb{Z}} f((\chi_m + k)/y_{m}).$$
Let $s_m$ be the supremum of $y_{N,m}$ for $N \geq 1$. For $|k| > A + 1$ or $A s_m < \chi_m < 1 - As_m$,
one has: 
$$\mathds{1}_{y_{N,m} > 0} f((\chi_m + k)/y_{N,m}) =  f((\chi_m + k)/y_{m}) = 0,$$
and since 
$$\mathbb{E} \left[ \sum_{m \geq 1} s_m \right] < \infty,$$
there exists a.s. a (random) finite subset $S$ of $\mathbb{N}^* \times \mathbb{Z}$ such that:
$$\int_{\mathbb{R}} f \, d\bar{\tau}_{N} 
= \sum_{(m,k) \in S} \mathds{1}_{y_{N,m} > 0} f((\chi_m + k)/y_{N,m})$$
and $$\int_{\mathbb{R}} f \, d\bar{\tau}_{\infty} 
= \sum_{(m,k) \in S} f((\chi_m + k)/y_{m}).$$
Since a.s., for all $m \geq 1$, $k \in \mathbb{Z}$:
$$ \mathds{1}_{y_{N,m} > 0} f((\chi_m + k)/y_{N,m})
 \underset{N \rightarrow \infty}{\longrightarrow} f((\chi_m+k)/y_m),$$
one deduces Proposition \ref{bartau}.
\end{proof} 
\noindent
One can now study the correlation measures associated with the point processes studied above. 
More precisely, if 
$$\tau_N (M_N) = \sum_{j \geq 1} \delta_{x_j},$$
where $(x_j)_{j \geq 1}$ is a sequence of real numbers, one can define, for all $q \geq 1$ the random measure on 
$\mathbb{R}^q$:
$$\tau_N^{[q]} (M_N) = \sum_{j_1 \neq j_2 \neq ... \neq j_q} \delta_{(x_{j_1},...,x_{j_q})}.$$
Similarly, one can define $\tau_{\infty}^{[q]} (\theta, \mathcal{L})$,
$\bar{\tau}^{[q]}_{N}$ and $\bar{\tau}^{[q]}_{\infty}$. Of course, one has the following equalities in 
distribution:
$$\tau_N^{[q]} (M_N) = \bar{\tau}^{[q]}_{N}$$
and 
$$\tau_{\infty}^{[q]} (\theta, \mathcal{L}) = \bar{\tau}^{[q]}_{\infty}.$$
Proposition \ref{bartau} can be generalized as follows:
\begin{proposition} \label{bartauq}
For all $q \geq 1$ the random measure $\bar{\tau}^{[q]}_{N}$  converges vaguely towards
$\bar{\tau}^{[q]}_{\infty}$, which is locally finite. 
\end{proposition}
\begin{proof}
The fact that $\bar{\tau}^{[q]}_{\infty}$ is locally finite is a consequence of
the local finiteness of $\bar{\tau}_{\infty}$. 
Now, let $f$ be a continuous function from $\mathbb{R}^q$ to $\mathbb{R}$, 
with support included in $[-A,A]^q$ (for some $A > 0$). One has:
$$\int_{\mathbb{R}} f \, d\bar{\tau}^{[q]}_{N} =
  \sum_{(m_1,k_1) \neq ... \neq (m_q,k_q) \in \mathbb{N}^* \times \mathbb{Z}}
\mathds{1}_{y_{N,m_1},...,y_{N,m_q} > 0} f((\chi_{m_1} + k_1)/y_{N,m_1},..., (\chi_{m_q} + k_q)/y_{N,m_q})$$
and 
$$\int_{\mathbb{R}} f \, d\bar{\tau}^{[q]}_{\infty} =
  \sum_{(m_1,k_1) \neq ... \neq (m_q,k_q) \in \mathbb{N}^* \times \mathbb{Z}}
 f((\chi_{m_1} + k_1)/y_{m_1},..., (\chi_{m_q} + k_q)/y_{m_q})$$
Let $s_m$ be the supremum of $y_{N,m}$ for $N \geq 1$. If for some $j \leq q$,
 $|k_j| > A + 1$ or $A s_{m_j} < \chi_{m_j} < 1 - As_{m_j}$,
one has: 
$$\mathds{1}_{y_{N,m_1},...,y_{N,m_q} > 0} f((\chi_{m_1} + k_1)/y_{N,m_1},..., (\chi_{m_q} + k_q)/y_{N,m_q})=0$$
and
$$  f((\chi_{m_1} + k_1)/y_{m_1},..., (\chi_{m_q} + k_q)/y_{m_q})= 0.$$
Since
$$\mathbb{E} \left[ \sum_{m \geq 1} s_m \right] < \infty,$$
there exists a.s. a (random) finite subset $S$ of $\mathbb{N}^* \times \mathbb{Z}$ such that:
$$\int_{\mathbb{R}} f \, d\bar{\tau}^{[q]}_{N} =
  \sum_{(m_1,k_1) \neq ... \neq (m_q,k_q) \in S}
\mathds{1}_{y_{N,m_1},...,y_{N,m_q} > 0} f((\chi_{m_1} + k_1)/y_{N,m_1},..., (\chi_{m_q} + k_q)/y_{N,m_q})$$
and 
$$\int_{\mathbb{R}} f \, d\bar{\tau}^{[q]}_{\infty} =
  \sum_{(m_1,k_1) \neq ... \neq (m_q,k_q) \in S}
 f((\chi_{m_1} + k_1)/y_{m_1},..., (\chi_{m_q} + k_q)/y_{m_q})$$
Since a.s., for all $m_1,...,m_q \geq 1$, $k_1,...,k_q \in \mathbb{Z}$:
$$ \mathds{1}_{y_{N,m_1},...,y_{N,m_q} > 0} f((\chi_{m_1} + k_1)/y_{N,m_1},..., (\chi_{m_q} + k_q)/y_{N,m_q})$$
tends to $$f((\chi_{m_1} + k_1)/y_{m_1},..., (\chi_{m_q} + k_q)/y_{m_q})$$ when $N$ goes to infinity,
one deduces Proposition \ref{bartauq}.
\end{proof}
\noindent
Now, for $q \geq 1$, we define the "$q$-correlation measure" $\tilde{\tau}^{[q]}_N$
 associated with the point process $\tau_{N} (M_N)$ (or equivalently, 
$\bar{\tau}_N$), as the average of the measure $\tau_{N}^{[q]} (M_N)$ (or $\bar{\tau}^{[q]}_N$). 
Similarly, we define the measure $\tilde{\tau}_{\infty}^{[q]}$ as the average of 
$\tau_{\infty}^{[q]}(\theta, \mathcal{L})$ of $\bar{\tau}^{[q]}_{\infty}$. 
Both $\tilde{\tau}_N^{[q]}$ and $\tilde{\tau}_{\infty}^{[q]}$ are positive measures on $\mathbb{R}^q$. 
From the convergence of $\bar{\tau}^{[q]}_N$ toward $\bar{\tau}^{[q]}_{\infty}$, one deduces the following 
result:
\begin{proposition} \label{tildetau}
For all $q \geq 1$, the measures $\tilde{\tau}_N^{[q]}$ ($N \geq 1$) and $\tilde{\tau}_{\infty}^{[q]}$
are locally finite, and $\tilde{\tau}_N^{[q]}$ converges vaguely towards $\tilde{\tau}_{\infty}^{[q]}$ when 
$N$ goes to infinity. 
\end{proposition}
\begin{proof}
Let $f$ be a continuous function from $\mathbb{R}$ to $\mathbb{R}$, with compact support. One has, 
$$\left|\int_{\mathbb{R}} f \, d\bar{\tau}^{[q]}_{N} \right| \leq
  ||f||_{\infty} \, \sum_{(m_1,k_1) \neq ... \neq (m_q,k_q) \in \mathbb{N}^* \times \mathbb{Z}}
\mathds{1}_{(|k_j| \leq A+1)_{j \leq q}, \left( \chi_{m_j} \notin (As_{m_j}, 1 - As_{m_j})\right)_{j \leq q}}.
$$
which is independent of $N$ and has a finite expectation. Indeed, this quantity is bounded by a constant times
$$\left[\sum_{m=1}^{\infty} \left( \mathds{1}_{\chi_m \notin (As_m, 1-As_m)} \right) 
\right]^q,$$
and 
\begin{align*}
\mathbb{E} \left[ \left(\sum_{m=1}^{\infty} \left( \mathds{1}_{\chi_m \notin (As_m, 1-As_m)} \right) 
\right)^q \, \right] & \leq \sum_{m_1,...,m_q \geq 1} \mathbb{P}[\left(\chi_{m_j}
 \notin (As_j, 1-As_j)\right)_{j \geq 1}]
\\ & \leq \sum_{m \geq 1} \sum_{\max \{m_j, 1 \leq j \leq q \} = m} 
\mathbb{P}[\chi_m \notin (As_m, 1-As_m) ]
\\ & \leq 2A \sum_{m \geq 1} m^q \mathbb{E}[s_m] < \infty,
\end{align*}
since the expectation of $s_m$ decreases exponentially with $m$ (see the proof of \eqref{sm}). 
Almost surely, by proposition \ref{bartau},
$$\int_{\mathbb{R}} f \, d\bar{\tau}^{[q]}_{N} \underset{N \rightarrow \infty}{\longrightarrow}
\int_{\mathbb{R}} f \, d\bar{\tau}^{[q]}_{\infty}, $$
and then one obtains, by taking  expectation and applying dominated convergence:
$$\int_{\mathbb{R}} f \, d\tilde{\tau}^{[q]}_{N} \underset{N \rightarrow \infty}{\longrightarrow}
\int_{\mathbb{R}} f \, d\tilde{\tau}^{[q]}_{\infty}, $$
where all these integrals are finite. This proves Proposition \ref{tildetau}.
\end{proof}
\noindent
For $q = 1$, the $q$-correlation measure of a subset of $\mathbb{R}$ is simply the average number
of points lying on this set. It can be very simply expressed:
\begin{proposition} \label{1correlation}
The 1-correlation measures $\tilde{\tau}^{[1]}_N$ ($N \geq 1$) and $\tilde{\tau}^{[1]}_{\infty}$ are equal to 
Lebesgue measure on $\mathbb{R}$.
\end{proposition}
\begin{proof}
Let $f$ be a nonnegative, continuous function from $\mathbb{R}$ to $\mathbb{R}$. One has:
\begin{align*}
\mathbb{E} \left[ \left. \int_{\mathbb{R}} f \, d\bar{\tau}_N \, \right| \, (y_{N,m})_{m \geq 1}  \right] 
& =
 \sum_{m \geq 1} \mathds{1}_{y_{N,m} > 0} \sum_{k \in \mathbb{Z}} \int_0^1 f((k+x)/y_{N,m}) dx
\\ & = \sum_{m \geq 1} y_{N,m} \int_{\mathbb{R}} f(z) dz \\ & = 
\int_{\mathbb{R}} f(z) dz,
\end{align*}
\noindent
which proves Proposition \ref{1correlation} for $\tilde{\tau}^{[1]}_N$. The proof for $\tilde{\tau}^{[1]}_{\infty}$
is similar.
\end{proof}
\noindent
We remark that for $q \geq 2$, the correlation measure $\tilde{\tau}^{[q]}_N$ is not absolutely continuous
with respect to the Lebesgue measure. Indeed, for all integers $l$, $1 \leq l \leq N$,
the probability that the point process associated with $\bar{\tau}_N$ has two points separated
by an interval of exactly $N/l$ is not equal to zero (this event holds if the corresponding permutation
has a cycle of length $l$). Similarly, for $q \geq 3$, $\tilde{\tau}^{[q]}_{\infty}$ is not absolutely
continuous with respect to the Lebesgue measure, since the point process associated with $\bar{\tau}_{\infty}$
has almost surely three distinct points $x, y, z \in \mathbb{R}$ such that $y-x=z-y$. 

However, despite the fact that $\tilde{\tau}^{[2]}_N$ is not absolutely continuous with respect to the Lebesgue 
measure for all $N \geq 1$, its limit $\tilde{\tau}^{[2]}_{\infty}$ is absolutely continuous with
respect to the Lebesgue measure. More precisely one has the following: 
\begin{proposition} \label{2correlation}
The measure $\tilde{\tau}^{[2]}_{\infty}$ on $\mathbb{R}^2$ (which depends on the parameter $\theta$) 
has a density $\rho$  with respect to the Lebesgue measure, which is called "2-correlation function", and which is
given by: $$\rho(x,y) = \phi_{\theta}(x-y),$$
where the function $\phi_{\theta}$ from $\mathbb{R}$ to $\mathbb{R}$ is defined by:
$$\phi_{\theta} (x) = \frac{\theta}{\theta+1} + \frac{\theta}{x^2} 
\sum_{a \in \mathbb{N}^*, a \leq |x|} a \left(1 - \frac{a}{|x|} \right)^{\theta-1}. $$
\end{proposition}
\begin{proof}
One can write:
\begin{align*}
\bar{\tau}^{[2]}_{\infty} & = \sum_{m \neq m' \geq 1} \sum_{k, k' \in \mathbb{Z}}
\delta_{\left( (\chi_m + k)/y_{m}, (\chi_{m'} + k')/y_{m'} \right)} \\
& + \sum_{m \geq 1} \sum_{k \neq k' \in \mathbb{Z}} \delta_{\left( (\chi_m + k)/y_{m}, (\chi_m + k')/y_{m}
\right)}, 
\end{align*}
which implies that for all nonnegative and continuous functions $f$ from $\mathbb{R}^2$ to $\mathbb{R}$, 
\begin{align}
\mathbb{E} \left[ \left. \int_{\mathbb{R}^2} f \, d\bar{\tau}^{[2]}_{\infty}  \, \right| \, (y_m)_{m \geq 1}
\right]
& = \sum_{m \neq m' \geq 1} \sum_{k, k' \in \mathbb{Z}}
\int_0^1 \int_0^1 f\left( (k+x)/y_{m}, (k'+x')/y_{m'} \right) \, dx \, dx' \nonumber \\
& + \sum_{m \geq 1} \sum_{k \neq k' \in \mathbb{Z}} \int_0^1 f\left( (k+x)/y_m, (k'+x)/y_m \right) \, dx 
\nonumber \\ 
& = \left( \sum_{m \neq m' \geq 1}  y_m y_{m'} \right) \, \int_{\mathbb{R}^2} f(x,x') \, dx \, dx'  \nonumber \\ 
& + \sum_{m \geq 1} \sum_{a \in \mathbb{Z} \backslash \{0\}} 
\, y_m \int_{\mathbb{R}} f(x, x + a/y_m) \, dx. \label{dru}
\end{align}
\noindent
The expectation of the first term of \eqref{dru} is equal to 
$$ \left(1 - \mathbb{E} \left[\sum_{m \geq 1} y_m^2 \right] \right) \, 
\int_{\mathbb{R}^2} f(x,x') \, dx \, dx' = \frac{\theta}{\theta + 1} \, 
\int_{\mathbb{R}^2} f(x,x') d x d x'.$$
In order to compute the expectation of the second term, let us fix $m$ and $a$, and let us denote by $d_m$ the 
density of the probability distribution of $y_m$ (recall that $d_m(x) = 0$ for all $x \notin [0,1]$). 
One has:
\begin{align*}
\mathbb{E} [y_m f(x, x+a/y_m)] & = \int_{0}^{1} t d_m(t)  f(x, x + a/t) dt \\ 
& = \int_{\mathbb{R}} \frac{a^2}{|u|^3} d_m (a/u) f(x, x+u) du
\end{align*}
\noindent
and then:
$$\mathbb{E} \left[  y_m \int_{\mathbb{R}} f(x, x + a/y_m) \, dx \right] 
= \int_{\mathbb{R}^2} \frac{a^2}{|x'-x|^3} d_m (a/(x'-x)) f(x, x') dx dx'.$$
Finally, the expectation of the second term of \eqref{dru} is:
$$ \sum_{a \in \mathbb{Z} \backslash \{0\}} 
\sum_{m \geq 1} \int_{\mathbb{R}^2} \frac{a^2}{|x'-x|^3} d_m (a/(x'-x)) f(x, x') dx dx'$$
which proves Proposition \ref{2correlation} with 
$$\phi_{\theta} (x) = \frac{\theta}{\theta+1} + \frac{1}{|x|^3} 
\sum_{a \geq 1} a^2 \sum_{m \geq 1} d_m (a/|x|).$$
Now, for all continuous functions $g$ from $[0,1]$ to $\mathbb{R}_+$, one has:
$$\mathbb{E} \left[ \sum_{m \geq 1} y_m g(y_m) \right]= \mathbb{E} [g(y_M)],$$
where conditionally on $(y_m)_{m \geq 1}$, the random index $M$ is chosen in a size-biased way, i.e. $M=m$
with probability $y_m$. By classical properties of GEM and Poisson-Dirichlet processes, $y_M$ is equal to
$y_1$ in distribution, and then its density at $x \in [0,1]$ is $\theta (1-x)^{\theta-1}$.
Hence, one deduces that
$$\int_{0}^1 x g(x) \sum_{m \geq 1} d_m(x) = \theta \int_0^1 (1-x)^{\theta-1} g(x) \, dx,$$
and then, for almost every $x \in [0,1]$,
$$\sum_{m \geq 1} d_m(x) = \frac{\theta (1-x)^{\theta-1}}{x},$$
which implies Proposition \ref{2correlation}. 
\end{proof}
\subsection{The smallest eigenangle}
Another interesting problem about the point process associated with $\tau_{N} (M_N)$ is the estimation of 
its smallest positive point. This point corresponds (after scaling the eigenangle by $N$) to
 the first eigenvalue of $M_N$ obtained by starting from
1 and by turning counterclockwise on the unit circle. Its distribution has a limit by the following
result:
\begin{proposition} \label{smallestpoint}
With the notation above, the smallest positive point corresponding to the random measure $\bar{\tau}_N$
tends a.s. to the smallest positive point corresponding to $\bar{\tau}_{\infty}$.
\end{proposition}
\begin{proof}
One has, for all $m \geq 1$, $$\chi_m/y_{N,m} \underset{N \rightarrow \infty}{\longrightarrow}
\chi_m/y_m,$$
if by convention, $\chi_m/y_{N,m} = + \infty$ for $y_{N,m}=0$.
One deduces that for all $m_0 \geq 1$,
\begin{align*} \label{smallestpoint}
\underset{N \rightarrow \infty}{\lim \, \sup} \, \inf \{\chi_m/y_{N,m}, m \geq 1\}
& \leq \underset{N \rightarrow \infty}{\lim \, \sup} \, \inf \{\chi_m/y_{N,m}, 1 \leq m \leq  m_0\} \\
& = \inf \{\chi_m/y_{m}, 1 \leq m \leq  m_0\} 
\end{align*}
\noindent
and then, by taking $m_0\to\infty$:
$$\underset{N \rightarrow \infty}{\lim \, \sup} \, \inf \{\chi_m/y_{N,m}, m \geq 1\}
\leq \inf \{\chi_m/y_{m}, m \geq  1\}.$$
On the other hand, for all $A > 0$, there exists a.s. $m_1 \geq 1$ such that $\chi_m/s_m \geq A$ for 
all $m \geq m_1$, which implies that $\chi_m/y_{N,m} \geq A$ and $\chi_m/y_m \geq A$. Consequently
\begin{align*}
\underset{N \rightarrow \infty}{\lim \, \inf} \, \inf \{\chi_m/y_{N,m}, m \geq 1\}
& \geq A \, \wedge \, \underset{N \rightarrow \infty}{\lim \, \inf} \, \inf \{\chi_m/y_{N,m}, 1 \leq m \leq m_1\}
\\ & \geq A \, \wedge \,  \inf \{\chi_m/y_{m}, 1 \leq m \leq m_1\}
\\ & \geq A \, \wedge \, \inf \{\chi_m/y_m, m \geq 1 \}. 
\end{align*}
\noindent
By taking $A \rightarrow \infty$, one obtains:
$$\underset{N \rightarrow \infty}{\lim \, \inf} \, \inf \{\chi_m/y_{N,m}, m \geq 1\}
\geq \inf \{\chi_m/y_m, m \geq 1 \},$$
and finally, 
$$\inf \{\chi_m/y_{N,m}, m \geq 1\} \, \underset{N \rightarrow \infty}{\longrightarrow}
\, \inf \{\chi_m/y_m, m \geq 1 \},$$
which proves Proposition \ref{smallestpoint}. 
\end{proof}
\noindent
One immediately deduces the following:
\begin{corollary}
The smallest positive point of the random measure $\tau_{N}(M_N)$
converges in distribution to the smallest positive point of $\tau_{\infty} (\theta, \mathcal{L})$.
\end{corollary}
\noindent
At this stage, one naturally seeks for the explicit computation of the distribution
of the smallest positive point of $\tau_{\infty}(\theta, \mathcal{L})$. 
We can remark the similarity between this process and the limit point process 
(a determinantal process with sine kernel) obtained from the scaled eigenangles
(with the same scaling by $N$)
 of a random unitary matrix following the Haar measure. For the Haar measure, 
the law of the smallest positive point satisfies a certain Painlevé-V differential equation (see for instance \cite{agz} for more details and references). In our
 case, one can also obtain this probability distribution as a solution of some integral equation. More precisely, one has the following result:
 \begin{proposition} \label{nnkd}
For all $x \geq 0$, let $G(x)$ be the probability that the point process $\tau_{\infty} (\theta, \mathcal{L})$ has no point in the interval $(0,x)$, and for all $x \in \mathbb{R}$, let 
us set:
$$H(x) := \mathds{1}_{x > 0} x^{\theta -1} G(x) .$$
Then $H$ is integrable and satisfies the following equation:
\begin{equation}
 x H(x) = \theta \, \int_0^1 (1-y) H(x-y) \, dy. \label{kdnq}
 \end{equation}
Moreover, if the Fourier transform $\widehat{H}$ of $H$ is given by
	$$ \widehat{H}(\lambda) = \int_{-\infty}^{\infty} e^{- i \lambda x} H(x) \, dx,$$
	then it satisfies the equation:
	\begin{equation}
	\widehat{H} (\lambda)  = \widehat{H} (0) \, \exp \left(-i\theta \int_0^{\lambda}  \, \frac{1 - e^{-i \mu} -i \mu}{\mu^2} \, d \mu \right), \label{ptrbl}
	\end{equation}
	for all $\lambda \in \mathbb{R}$.
 \end{proposition}
\begin{proof}
The probability $G(x)$ can be expressed as follows:
$$G(x) = \mathbb{P} \left[ \inf_{m \geq 1}\{ \chi_m/ y_m \} \geq x \right],$$
where we recall that $(y_m)_{m \geq 1}$ is a GEM process of parameter $\theta$. 
Conditionally on $y_1$, the sequence $(y'_m := y_{m+1}/(1-y_1))_{m \geq 1}$ is also a GEM
process of parameter $\theta$. Therefore, 
\begin{align*}
\mathbb{P} \left[ \inf_{m \geq 2} \{ \chi_m/y_m \} \geq x \, | \, \chi_1, y_1 \, \right] 
& = \mathbb{P} \left[ \inf_{m \geq 1} \{ \chi_{m + 1} / y'_{m} \} \geq x (1-y_1) \, | \chi_1, y_1 \, \right]
= G(x (1-y_1)).
\end{align*}
By taking the expectation, one obtains:
\begin{align*}
G(x) & = \mathbb{E} \left[ \mathds{1}_{\chi_1/y_1 \geq x} \, G(x(1-y_1)) \right] \\
& = \mathbb{E} \left[ (1- xy_1)_+ G(x(1-y_1)) \right] \\
& = \int_0^{1 \wedge 1/x} \theta (1-y)^{\theta  - 1} (1 - yx) G (x(1-y)) \, dy 
\end{align*}
since the law of $y_1$ has density $\theta (1-y)^{\theta-1}$ with respect to Lebesgue measure. Hence:
$$ G(x) = \frac{\theta}{x^{\theta}} \int_{0}^{1} (x-y)^{\theta - 1} (1-y) G(x-y) \, dy$$
where by convention, we set $G(y) = 0$ for all strictly negative $y$. 
This implies the equation \eqref{kdnq} in Proposition \ref{nnkd}. Note that for all $x \geq 0$,
$$H(x) \leq x^{\theta-1}.$$
Moreover, for $x \geq 2$, $0 \leq y \leq 1$:
\begin{align*}
H(x-y) & = (x-y)^{\theta-1} G(x-y) \leq (x-y)^{\theta-1} G(x-1) \\
& \leq (x-1)^{\theta-1} 2^{(\theta-1)_+}  G(x-1) \\ & \leq 2^{(\theta-1)_+}  H(x-1),
\end{align*}
and by \eqref{kdnq},
$$ H(x) \leq \frac{ 2^{(\theta-1)_+} \theta  }{2x} H(x-1).$$
Hence, at infinity, $H$ decreases faster than exponentially, and $H$, $x \mapsto xH(x)$
 are in $L^1$. One deduces that $\widehat{H}$,
 is well-defined and differentiable. By \eqref{kdnq} one has for all $\lambda \in \mathbb{R}$:
$$\widehat{H}'(\lambda) = - i \theta \widehat{H} (\lambda) \widehat{K} (\lambda),$$
where $K$ is the function defined by $K(y) = (1-y) \mathds{1}_{0 \leq y \leq 1}$.
Therefore:
$$\widehat{H} (\lambda)  = \widehat{H} (0) \, \exp \left( -i \theta \int_0^\lambda \hat{K} (\mu) \, d \mu \right),$$
which implies \ref{ptrbl}.
\end{proof}

\providecommand{\bysame}{\leavevmode\hbox to3em{\hrulefill}\thinspace}
\providecommand{\MR}{\relax\ifhmode\unskip\space\fi MR }
\providecommand{\MRhref}[2]{%
  \href{http://www.ams.org/mathscinet-getitem?mr=#1}{#2}
}
\providecommand{\href}[2]{#2}

\end{document}